\documentclass[reqno,twoside,11pt]{amsart}

\usepackage{amsmath,amsfonts,calrsfs,fullpage,amssymb,color,verbatim,eucal,yfonts,mathrsfs}

\usepackage{amsthm}
\usepackage{xpatch}

\xpatchcmd{\proof}{\itshape}{\prooflabelfont}{}{}
\newcommand{\prooflabelfont}{\bfseries}

\makeatletter
\newcommand\thankssymb[1]{\textsuperscript{\@fnsymbol{#1}}}

\makeatother

\newtheorem{Theorem}{Theorem}[section]
\newtheorem{Definition}{Definition}
\newtheorem{Proposition}[Theorem]{Proposition}
\newtheorem{Lemma}[Theorem]{Lemma}

\newtheorem{Remark}[Theorem]{Remark}

\newtheorem{Hypothesis}{Hypothesis}
\makeatletter
\@addtoreset{equation}{section}

\makeatother

\def\a{\alpha}

\def\la{\lambda}

\def\R{\mathbb R}
\def\N{\mathbb N}

\def\E{\mathbb E}
\def\P{\mathbb P}
\def\S{\mathbb S}
\def\F{\mathcal{F}}
\def\Fs{\mathcal{F}_s}
\def\O{\mathcal O}
\def\<{\langle}
\def\>{\rangle}
\def\Oc{\bar{\mathcal{O}}}

\def\ds{\displaystyle}
\def\e{\epsilon}

\xpatchcmd{\proof}{\itshape}{\prooflabelfont}{}{}
\DeclareMathOperator{\sign}{sign}

\title{Nonlinear random perturbations of  Reaction Diffusion Equations}

\author[S. Cerrai]{Sandra Cerrai\thankssymb{1}}
\address[S. Cerrai]{Department of Mathematics\\
University of Maryland\\ 
4176 Campus Drive, College Park, MD 20742, United States}
\thanks{\thankssymb{1} Partially supported by the
NSF Grant DMS-2348096 (2024-2027), {\em 
 Multiscale Analysis of Infinite-Dimensional Stochastic Systems}}
\email{cerrai@umd.edu}

\author[G. Guatteri]{Giuseppina Guatteri\thankssymb{2}}
\address[G. Guatteri]{Dipartimento di Matematica F. Brioschi\\
Politecnico di Milano\\
via Bonardi 9, 20133 Milano, Italy}
\thanks{\thankssymb{2}Partially supported by GNAMPA Grant, {\em Riduzione di modelli in sistemi stocastici in dimensione infinita a due scale temporali}}
\email{giuseppina.guatteri@polimi.it}

\author[G. Tessitore]{Gianmario Tessitore\thankssymb{2}}
\address[G. Tessitore]{
Dipartimento di Matematica e Applicazioni\\
Universit\`a di Milano-Bicocca\\
Via Roberto Cozzi 55, 20125 Milano, Italy}
\email{gianmario.tessitore@unimib.it}

\subjclass[2010]{}

\begin{document}

 \begin{abstract} 
 This paper investigates the well-posedness and small-noise asymptotics of a class of stochastic partial differential equations defined on a bounded domain of $\mathbb{R}^d$, where the diffusion coefficient depends nonlinearly and non-locally on the solution through a conditional expectation. The reaction term is assumed to be merely continuous and to satisfy a quasi-dissipativity condition, without requiring any growth bounds or local Lipschitz continuity. This setting introduces significant analytical challenges due to the temporal non-locality and the lack of regularity assumptions. Our results represent a substantial advance in the study of nonlinear stochastic perturbations of SPDEs, extending the framework developed in \cite{CGT}.

 \end{abstract}

 \maketitle
 
\tableofcontents
 
\section{Introduction} 
\label{sec2}

The study of stochastic perturbations of deterministic dynamical systems has long been a central theme in the analysis of random phenomena across infinite-dimensional settings, particularly in the context of stochastic partial differential equations (SPDEs). When these perturbations are small, large deviations theory provides a powerful framework for understanding the asymptotic behavior of trajectories and the structure of rare events. While this theory has been extensively developed in both finite and infinite dimensions in the linear case, its extension to quasi-linear frameworks remains a challenging and largely open area of research.

In their paper \cite{FK},  Freidlin and Koralov  consider a class of stochastic perturbations of finite-dimensional dynamical systems, and study the long-time behavior of associated quasi-linear parabolic equations of the form

$$
\partial_t u^\epsilon(t,x) = \frac{\epsilon}{2} \sum_{i,j=1}^d a_{ij}(x,u^\epsilon(t,x)) \partial_{ij} u^\epsilon(t,x) + \sum_i b_i(x) \partial_i u^\epsilon(t,x),
 $$
 coupled with the backward interpretation of the solution via a nonlinearly perturbed diffusion\footnote{\ Here, for every $x \in\,\mathbb{R}^d$, we denoted by $a(x)$ the matrix $(\sigma\sigma^\star)(x)$.}
$$
dX_\epsilon(s) = b(X_\epsilon(s)) ds + \sqrt{\epsilon}\, \sigma(X_\epsilon(s), u^\epsilon(t-s, X_\epsilon(s))) dB_s,
$$
which leads naturally to a nonlinear semigroup formulation. Their analysis relies on classical PDE techniques in finite dimensions and yields sharp asymptotic characterizations of solutions on exponential time scales via the quasi-potential associated with the large deviations principle.

\smallskip

Our previous paper \cite{CGT} represents the first step toward the extension of this theory to infinite-dimensional settings, particularly separable Hilbert spaces. In \cite{CGT} we were concerned with the analysis of quasi-linear parabolic PDEs with small viscosity, where the second-order term in the PDE is nonlinearly dependent on the solution itself. The corresponding stochastic dynamics are governed by SPDEs with small  multiplicative noise and feedback through the solution of the PDE. In contrast to the classical (linear) SPDE-PDE duality, the coupling here is fully nonlinear and requires interpreting the SPDE in terms of a forward-backward stochastic differential system.

Our primary goal was studying the well-posedness of quasi-linear PDEs of the form

\begin{equation}\label{cgt3-intro}
\partial_t u^\epsilon(t,x) = \frac{\epsilon}{2} \operatorname{Tr}[\sigma\sigma^\star(x,u^\epsilon(t,x)) D^2_x u^\epsilon(t,x)] + \langle Ax + f(x), D_x u^\epsilon(t,x) \rangle_H,
\end{equation}
defined on a Hilbert space $H$, and relating this PDE to the stochastic evolution equation
\begin{equation}\label{cgt2-intro} dX_\epsilon(s) = (AX_\epsilon(s) + f(X_\epsilon(s))) ds + \sqrt{\epsilon} \,\sigma(X_\epsilon(s), u^\epsilon(t-s, X_\epsilon(s))) dW_s,
 \end{equation}
where $W_s$, $s \in\,[0,t]$, is a space-time white noise, defined on the stochastic basis $(\Omega, \mathcal{F}, \{\mathcal{F}_s\}_{s \in\,[0,t]}, \mathbb{P})$. As in \cite{FK}, we showed that the nonlinear PDE and the SPDE are connected through a nonlinear Feynman-Kac-type formula, in the form

 $$
 u^\epsilon(t,x) = \mathbb{E}[g(X^\epsilon(t,x))] =: T^\epsilon_t g(x),
 $$
where $T^\epsilon_t$ is now a nonlinear semigroup.

 A major technical challenge in the infinite-dimensional setting lies in the lack of a comprehensive theory of deterministic quasi-linear PDEs in Hilbert spaces. Unlike the finite-dimensional case, where viscosity solutions or Schauder estimates can be employed to establish classical solvability, the infinite-dimensional case requires more delicate functional analytic techniques. We therefore built on the abstract framework for infinite-dimensional Kolmogorov equations developed in \cite{DPZ2}, and extended it to deal with nonlinear second-order terms of the form 
 \begin{equation}\label{cgt1-intro}\sigma^\star\sigma(x,r)=Q+\delta\,\lambda(x,r),\ \ \ \ x  \in\,H,\ \ r \in\,\mathbb{R}
\end{equation}
for some bounded and non-negative symmetric  operator $Q$, a  continuous mapping $\lambda$ defined on $H\times \mathbb{R}$ with values in the space of trace-class operators and a constant $\delta>0$, which we took small enough.

Another contribution of \cite{CGT} was the derivation of a large deviations principle for the law of the trajectories of the solution $X^\epsilon$ of the stochastic system. Under the small-noise regime $\epsilon \downarrow 0$, we proved that the family of laws $\{\mathcal{L}(X^\epsilon)\}_{\epsilon > 0}$ satisfies a large deviations principle on $C([0,t];H)$ with an action functional defined in terms of a suitable controlled deterministic dynamics. This result forms the first step for future studies on exit times, metastability, and invariant measure concentration in infinite-dimensional systems with nonlinear feedback.

\smallskip

Now, as a consequence of  the Markov property, the following relation holds between the solution $u^\e$ of the quasi-linear PDE \eqref{cgt3-intro} and the solution $X^{t,x}_\e$ of the SPDE \eqref{cgt2-intro}
$$u_\e(t-s,X_\e^{t,x}(s)))=\mathbb{E}(g(X_\e^{t-s,y}(t-s)))\big \vert_{y=X_\e^{t,x}(s)}=\mathbb{E}(g(X_\e^{t,x}(t))\vert \mathcal{F}_s),$$
for every $s \in\,[0,t]$ and $x \in\,H$. In particular, 
equation \eqref{cgt2-intro} can be written as
\begin{equation}
\label{stoch-pde.intro1}
\left\{\begin{array}{l}
\ds{dX_\e^{t,x}(s)=\left(A	X_\e^{t,x}(s)+f(X_\e^{t,x}(s))\right)\,ds+\sqrt{\e}\,\sigma(X_\e^{t,x}(s),\mathbb{E}(g(X_\e^{t,x}(t))\vert \mathcal{F}_s))\,dW_s,}\\[10pt]
\ds{X_\e^{t,x}(0)=x.}
	\end{array}\right.
\end{equation}
In the new formulation \eqref{stoch-pde.intro1} there is no mention to $u^\e$, so that  we can study this class of nonlinear perturbations of non-linear PDEs directly, without coupling them to quasi-linear PDEs Hilbert spaces such as \eqref{cgt3-intro}. The advantage of this approach is that in our analysis we can now consider a much more general class of non-linearities $f$ and $\sigma$ and linear operators $A$ and we do not need any non-degenaracy condition on the noise, so that we can deal with any space dimension.

Although equation \eqref{stoch-fbsde.intro} is meaningful on its own and will, in fact, be analyzed directly without further transformation, it is insightful to interpret it within the framework of Forward-Backward Stochastic Differential Equations (FBSDEs). Specifically, when the noise is non-degenerate, one can define
\[
Y^{t,x}_\e(s) := \mathbb{E}\left(g(X_\varepsilon^{t,x}(t)) \mid \mathcal{F}_s\right),
\]
and thereby reformulate \eqref{stoch-pde.intro1} as the following coupled infinite-dimensional stochastic system:
\begin{equation}
\label{stoch-fbsde.intro}
\left\{
\begin{array}{l}
\displaystyle dX_\varepsilon^{t,x}(s) = \big(A X_\varepsilon^{t,x}(s) + f(X_\varepsilon^{t,x}(s))\big)\,ds + \sqrt{\varepsilon}\,\sigma(X_\varepsilon^{t,x}(s), Y^{t,x}_\e(s))\,dW_s, \quad 0 \leq s \leq t, \\[10pt]
\displaystyle -dY^{t,x}_\e(s) = Z^{t,x}_\e(s) Q^{-1/2} dW_s, \quad 0 \leq s \leq t, \\[10pt]
\displaystyle Y^{t,x}_\e(t) = g(X_\varepsilon^{t,x}(t)), \\[10pt]
\displaystyle X_\varepsilon^{t,x}(0) = x.
\end{array}
\right.
\end{equation}
Here, $Y_\epsilon$ and $Z_\epsilon$ take values in appropriate infinite-dimensional state spaces, and $Q$ denotes the covariance operator of the Wiener process $W_s$.

The solvability of strongly coupled forward-backward systems is not guaranteed, either in finite or infinite dimensions; see \cite{Antonelli} for counterexamples. In finite dimensions, well-posedness can be achieved, either via fixed-point arguments under restrictive quantitative conditions on the coefficients (see  \cite{PardouxTang}) or, in wider generality generally, by studying the associated Kolmogorov equations through the classical theory of PDEs, as exemplified by the celebrated {\em Four-Step Scheme} in \cite{MaProYong}, \cite{MaYong}. Other results in specific situations can be obtained by the {\em Continuation Method}, see \cite{Yong97}.

In infinite dimensions, the range of applicable methods is more limited, as many analytic tools do not carry over. In this context, energy estimate techniques under quantitative constraints remain the only viable approaches. 

In the present work, since our focus lies on the analysis of small perturbations, we exploit the small parameter $\varepsilon$ that naturally appears in the coupling term of the FBSDE system. When $\varepsilon$ is sufficiently small, the quantitative conditions stated in \cite{PardouxTang} are automatically satisfied. This allows us to establish the existence and uniqueness of a solution to equation \eqref{stoch-pde.intro1} by applying a contraction principle within a suitably chosen functional framework. It is important to note, however, that the admissible size of $\varepsilon$ may depend on the terminal time $t$.

In this paper, we  deal with a class of SPDEs on a bounded and smooth domain $\mathcal{O}\subset \mathbb{R}^d$, with $d\geq 1$, such as
\begin{equation}
\label{SPDE-intro}
\left\{\begin{array}{l}
\ds{\partial_s X(s,\xi)	=\mathcal{A}X(s,\xi)+f(X(s,\xi))+\sqrt{\epsilon}\,\sigma(X(s),\mathbb{E}\left(g(X(t))|\mathcal{F}_s)\right)(\xi)\,\partial_sW_s(s,\xi),}\\[10pt]
\ds{X(0,\xi)=x(\xi),\ \ \ \ \ \ \ X(s,\xi)=0,\ \ \xi \in\,\partial\mathcal{O},}	
\end{array}\right.
\end{equation}
where $\mathcal{A}$ is a uniformly elliptic second order differential operator, and $W_s$ is a cylindrical Wiener process in $L^2(\mathcal{O})$, which is white in time and appropriately colored in time, in case of dimension $d\geq 2$. As far as the coefficient $f$, we drop the Lipschitz-continuity condition we assumed in \cite{CGT}, and we assume that $f:\mathbb{R}\to \mathbb{R}$ is a continuous function such that
\begin{equation}\label{cgt10-intro}\begin{split}
      & f(u_2)-f(u_1) \leq k\, (u_2-u_1),\ \ \ \ \ \ \ \  u_1\leq u_2, 
    \end{split}	
	\end{equation}
	for some constant $k\in \mathbb{R}$. In particular, we do not even assume $f$ to be locally Lipschitz-continuity nor we impose any restriction on its growth. For the coefficients $g$ and $\sigma$, the model we have in mind is
	\[\sigma(x,y,r)(\xi)=\mathfrak{s}(\xi,x(\xi),y(\xi),r),\ \ \ \ \ \xi \in\,\bar{\mathcal{O}},\ \ \ \ x, y \in\,C(\bar{\mathcal{O}}),\ \ r \in\,\mathbb{R},\]
	and
	\[g(x)(\xi)=(\mathfrak{g}_1(\xi,x(\xi)),\mathfrak{g}_2(x)),\ \ \ \ \ \ \xi \in\,\bar{\mathcal{O}},\ \ \ \ x \in\,C(\bar{\mathcal{O}}),\]
for some continuous functions $\mathfrak{g}_1:\Oc\times \mathbb{R}\to\mathbb{R}$,   $\mathfrak{g}_2: C(\bar{\mathcal{O}})\to \mathbb{R}$, and $\mathfrak{s}:\Oc\times \mathbb{R}^3\to\mathbb{R}$, which satisfy suitable Lipschitz continuity properties.

A common strategy when dealing with deterministic or stochastic PDEs whose nonlinearities are only locally Lipschitz continuous and locally bounded is to localize the coefficients (see, e.g., \cite{CL}) or to construct solutions locally in time. In both approaches, the availability of uniform estimates for the solutions can then be used to establish global existence. However, this approach cannot be applied in the present setting due to the special structure of the diffusion coefficient
\[x \in\,C([0,t];C(\bar{\mathcal{O}}))\mapsto \sigma(x(s),\mathbb{E}(g(x(t))|\mathcal{F}_s)),\ \ \ \ \ \ s \in\,[0,t],\]
which is {\em non-local in time}, as it depends on the value of of the  $(x(s))_{s \in\,[0,t]}$ at a fixed (final) time $t>0$.
To address this difficulty, we adopt an alternative method inspired by a result introduced by M. Salins in \cite{Salins}, in the case of stochastic reaction diffusion defined on the whole space $\mathbb{R}^d$. In \cite{Salins}, for every $z\in C([0,t];C(\bar{\mathcal{O}}))$ Salins proposes to consider the equation
\begin{equation}
    \label{SalinsMild-intro}
    u(s,\xi)=\int_0^s S(s-r)f(u(r))(\xi)dr+z(s,\xi),\ \ \ \ \ \ \xi \in\,\bar{\mathcal{O}},
\end{equation}
where $S(s)$ denotes the transition semigroup associated with the realization of the operator $\mathcal{A}$, under Dirichlet boundary conditions. His key observation is that this equation  admits a unique solution  $\mathcal{M}_t(z) \in\,C([0,t];C(\bar{\mathcal{O}}))$, and -- despite  the non-Lipschitz nature of $f$ -- the mapping $\mathcal{M}_t$ is Lipschitz continuous on $C([0,t];C(\bar{\mathcal{O}}))$. Thus, defining for every $Z \in\, L^p(\Omega;C([0,t];C(\bar{\mathcal{O}})))$   
\[\gamma_t(Z)(s)=\int_0^s S(s-r) \sigma(Z(r),  \E(g(Z(t))|\F_r))\,dW_r,\ \ \ \ \ \ s \in\,[0,t],\]
in our case we obtain that a process $X \in\,C([0,t];C(\bar{\mathcal{O}}))$ is a mild solution to equation \eqref{SPDE-intro} if and only if it is a fixed point of the mapping
\[X \in\,C([0,t];C(\bar{\mathcal{O}}))\mapsto \mathcal{M}_t(\sqrt{\e}\,\gamma_t(X)+S(\cdot)x) \in\,C([0,t];C(\bar{\mathcal{O}})).\]
Therefore, if one can prove that  $\gamma_t$ is Lipschitz-continuous in  $L^p(\Omega;C([0,t];C(\bar{\mathcal{O}})))$, for sufficiently large  $p\geq 1$, then -- thanks  to the Lipschitz continuity of $\mathcal{M}_t$ --  the composed mapping is a contraction for small enough $\e>0$,  which ensures the existence and uniqueness of a mild solution to equation \eqref{SPDE-intro}.

\smallskip

In the same spirit of what we have done in \cite{CGT}, once the existence  and uniqueness of a mild solution $X^{t,x}_\e$ in $C([0,t];C(\bar{\mathcal{O}}))$ has been established, for every $x \in\,C(\bar{\mathcal{O}})$ and every $\epsilon\leq \e_t$ sufficiently small,  our goal is  to investigate the validity of a large deviation principle for the family  $\{\mathcal{L}(X^{t,x}_\e)\}_{\e \in\,(0,\e_t)}$. 
For every $\rho \in\,[0,t]$ and $y \in\,C(\bar{\mathcal{O}})$ we introduce the problem
\[
Y_\rho^y(s)=\int_\rho^s S(s-r)f(Y_\rho^y(r))dr
+ S(s-\rho)y, \ \ \ \ \ s\in\,[\rho,t].
\]
and we denote
\[ v_{t}(\rho,y) =  g(Y_\rho^{y}(t)),\ \ \ \ \ \ \rho \in\,[0,t],\ \ \ \ y \in\,C(\bar{\mathcal{O}}).\]
The mapping $Y^y_\rho$ depends continuously on $\rho \in\,[0,t]$. Moreover, the Lipschitz continuity of the mapping $\mathcal{M}_t$, as discussed above, implies that $Y^y_\rho$ depends Lipschitz continuously on the initial condition  $y \in\, C(\bar{\mathcal{O}})$, uniformly with respect to $\rho \in\,[0,t]$. Altogether, this implies that the mapping  $v_t:[0,t]\times C(\bar{\mathcal{O}})\to \mathbb{R}$ is continuous and Lipschitz-continuous in the second variable.
In particular, if for every $\varphi \in\,L^2(0,t;H_0)$ -- here $H_0$ is the reproducing kernel Hilbert space of the noise $W_s$ -- we introduce the problem
\begin{equation}
\label{sbm2-bis-main-intro}
\partial_s X(s)=\mathcal{A}	X(s)+f(X(s))+\sigma(X(s),  v_{t}(s,X(s)))\varphi(s),\ \ \ \ X(0)=x,	\end{equation}
we have that it admits a unique solution  $X^{t,x,\varphi} \in\,C([0,t];\bar{\mathcal{O}}))$. 
All this allows us to define
\[I_{t,x}(X):=\frac 12\,\inf\left\{\int_0^t\vert \varphi(s)\vert_{H_0}^2\,ds\,:\, X(s)=	X^{t,x,\varphi}(s),\ s \in\,[0,t]\right\}.\]
In what follows, we will prove that the family  
$\{\mathcal{L}(X^{t,x}_\e)\}_{\e \in\,(0,\e_t)}$ satisfies a large deviation principle in  $C([0,t];\bar{\mathcal{O}}))$, with  action functional $I_{t,x}$.

As well know, once the compactness of the level sets of $I_{t,x}$ is established, the large deviation principle is equivalent to the Laplace principle. Our strategy, therefore,  is to prove that the family  
$\{\mathcal{L}(X^{t,x}_\e)\}_{\e \in\,(0,\e_t)}$ satisfies a Laplace principle in  $C([0,t];\bar{\mathcal{O}}))$, with  action functional $I_{t,x}$.
To this purpose, following the general scheme in \cite{bdm} we consider the map $\mathcal{G}(x, \sqrt{\e} \,W)$ that connects the initial datum $x$ and the trajectories of the noise to the solution $X^{t,x}_{\e}$ of equation \eqref{stoch-pde.existence}. 
For every $t>0$ and  $M>0$, we introduce the set
\[\Lambda_{t,M}:=\left\{ \varphi \in\,\mathcal{P}_t\ :\ \vert \varphi\vert_{L^2(0,t;H_0)}\leq M,\ \mathbb{P}-\text{a.s.}\right\}\]
where $\mathcal{P}_t$ is the set of predictable processes in $L^2(\Omega;L^2(0,t;H_0))$.
Thus, given $\varphi \in \Lambda_{t,M}$ we  consider the process  $\mathcal{G}(x, \sqrt{\e}\, W+\int_0^{\cdot} \varphi(s) ds)$ with perturbed noise, that is the solution $X_{\e}^{t,x,\varphi}$ of the equation \begin{equation}
\label{stoch-pde.eps.LDP.1-intro}
\left\{\begin{array}{l}
\begin{array}{l}
\ds{\partial_s X_{\e}^{t,x,\varphi}(s)= \mathcal{A}	X_{\e}^{t,x,\varphi}(s)+F(X_{\e}^{t,x,\varphi}(s))}\\[10pt] 
\ds{ \quad \quad \quad \quad \quad \quad \quad \quad  +\sqrt{\e}\,\sigma(X_{\e}^{t,x,\varphi}(s),\mathbb{E}^{\e,\varphi}(g(X_{\e}^{t,x,\varphi}(t))|\mathcal{F}_s))\,\partial_s W_s^{\e, \varphi},}\\[8pt]
\ds{X_{\e}^{t,x,\varphi}(0)=x.}
	\end{array}	
\end{array}\right.
\end{equation}
where $\mathbb{E}^{\epsilon,\varphi}$ denotes the (conditional) expectation with respect to the probability $\mathbb{E}^{\epsilon, \varphi}$ under which  
\[W^{\e, \varphi}_s:=W_s+\frac{1}{\sqrt{\e} }\, \int_0^s\varphi(r)\,dr,\ \ \ \ \ \ s \in\,[0,t],\] is a Wiener process with covariance $Q$.
We notice that by the Markovianity and the uniqueness  of the solution to equation \eqref{SPDE-intro} we have
$$\mathbb{E}^{\e,\varphi}(g(X_{\e}^{t,x,\varphi}(t))|\mathcal{F}_s)=v_{t,\e}(s,X_{\e}^{t,x,\varphi}(s)),$$ where 
\begin{equation}\label{bms22-bis-intro}v_{t,\e}(\rho,y):=  \E(g(X^{t,y}_{\e,\rho}(t)),\ \ \ \ \ \ 0\leq \rho < t,\ \ \ y \in\,E,\end{equation}
 and $X^{t,y}_{\e,\rho}$ is the solution of \eqref{SPDE-intro} starting at time $\rho$ from $y$, namely
 \begin{equation*}
\begin{array}{l}
\ds{	X^{t,y}_{\e,\rho}(s)=S(s-\rho)y+\int_\rho^s S(s-r)f(X^{t,y}_{\e,\rho}(r))dr
}\\[14pt]
\ds{ \quad \quad \quad \quad \quad \quad \quad \quad  \quad +\sqrt{\e}\,\int_\rho^s S(s-r) \sigma(X^{t,y}_{\e,\rho}(r),  \E(g(X^{t,y}_{\e,\rho}(t))|\F_r))\,dW_r.}
\end{array}
\end{equation*} 
Consequently, for every $\varphi \in \Lambda_{t,M}$ and $\e>0$,  the process given by $\mathcal{G}(x, \sqrt{\e}\, W+\int_0^{\cdot} \varphi(s) ds)$  can be represented as the solution to the controlled equation
\begin{equation*}
\label{stoch-pde.eps.LDP-intro}
\left\{\begin{array}{l}
\begin{array}{l}
\ds{\partial_s X_{\e}^{t,x,\varphi}(s)= \mathcal{A}	X_{\e}^{t,x,\varphi}(s)+f(X_{\e}^{t,x,\varphi}(s))+ \sigma(X_{\e}^{t,x,\varphi}(s),v_{t,\e}(s,X_{\e}^{t,x,\varphi}(s)))\varphi(s)}\\[10pt] 
\ds{ \quad \quad \quad \quad \quad \quad \quad \quad \quad \quad \quad  +\sqrt{\e}\,\sigma(X_{\e}^{t,x,\varphi}(s),v_{t,\e}(s,X_{\e}^{t,x,\varphi}(s))) \,\partial_sW_s,}\\[8pt]
\ds{X_{\e}^{t,x,\varphi}(0)=x.}
	\end{array}	
\end{array}\right.
\end{equation*}
Then, the  Laplace principle will follows once we will prove that 
\[\lim_{\e\to 0} X_\e^{t,x,\varphi_{\e}}=X^{t,x,\varphi},\ \ \ \ \text{in distribution in}\ \ \ C([0,t],C(\bar{\mathcal{O}})),\]
for every sequence $\{\varphi_\epsilon\}_{\epsilon>0}$ in $ \Lambda_{t,M}$ such that 
\[\lim_{\e\to 0} \varphi_\e=\varphi,\ \ \ \ \text{in distribution in}\ \ \ L_w^2(0,t;H_0),\]
where we denote by $L^2_w(0,t;H_0)$ the space $L^2(0,T;H_0)$, endowed with the weak topology.

\section{Assumptions}
 In what follows, we will denote by $H$ the Hilbert space $L^2(\mathcal{O};\R)$, endowed with the scalar product
 \[\langle x,y\rangle_H=\int_\mathcal{O} x(\xi)y(\xi)\,d\xi,\]
 and the corresponding norm $\vert\cdot\vert_H$, 
and by $E$ the Banach space $C (\Oc; \R)$,  endowed with the sup-norm
\[|x|_E=\sup_{\xi \in\,\Oc}|x(\xi)|.\] We recall that here $\mathcal{O}\subset \R^d$ is a bounded open set with smooth boundary and $d\geq 1$.

\subsection{The coefficients}
\begin{Hypothesis}\label{H1}
\begin{enumerate}
\item[1.] The mapping  $\sigma: E\times E \times \mathbb{R}\to E$,  is Lipschitz continuous, with Lipschitz constant $L_\sigma$, and the mapping $g: E\rightarrow E \times \mathbb{R}$ is  Lipschitz continuous, with Lipschitz constant $L_{g}$.
\item[2.] The function $F$ is of the form
\[F(x)(\xi)=f(x(\xi)),\ \ \ \ \ \xi \in\,\bar{\mathcal{O}},\ \ \ x \in\,E,\] where $f: \mathbb{R} \rightarrow \mathbb{R}$ is a continuous function which verifies the following  condition
	\begin{equation}
	\label{ipotisi su f}
    \begin{split}
      & f(u_2)-f(u_1) \leq k\, (u_2-u_1),\ \ \ \ \ \ \ \  u_1\leq u_2, \\[10pt]
    \end{split}	
	\end{equation}
	for some constant $k\in \mathbb{R}$.

\end{enumerate}
\end{Hypothesis}

\begin{Remark}\label{rem4.1}
{\em  \begin{enumerate} \item[1.] If for every $x \in\,E$ we define 
\[g(x)(\xi)=(\mathfrak{g}_1(\xi,x(\xi)),\mathfrak{g}_2(x)),\ \ \ \ \ \ \ \ \ \ \xi \in\,\bar{\mathcal{O}},\]
for some continuous functions $\mathfrak{g}_1:\Oc\times \mathbb{R}\to\mathbb{R}$ which is Lipschitz continuous in the second variable, uniformly with respect to the first one, and $\mathfrak{g}_2: {E}\to \mathbb{R}$ which is Lipschitz, then it is immediate to check that  $g$  satisfies  condition 1. in Hypothesis \ref{H1}, as it maps $E$ into itself and is Lipschitz continuous. 
\item[2.] If we define
\[\sigma(x,y,r)(\xi):=\mathfrak{s}(\xi,x(\xi),y(\xi),r),\ \ \ \ \ \xi \in\,\bar{\mathcal{O}},\ \ \ x, y \in\,E,\]
for some continuous function $\mathfrak{s}:\Oc\times \mathbb{R}^3\to\mathbb{R}$, which is Lipschitz-continuous with respect to the last three variables, uniformly with respect to the first one, then $\sigma$ satisfies condition 1. in Hypothesis \ref{H1}.
\item[.] Given the mapping $\sigma:E\times E\times \mathbb{R}\to E$ we can define a new mapping, which with some abuse of notation we will still denote by $\sigma$, by setting
\[[\sigma(x,y,r)z](\xi)=\sigma(x,y,r)(\xi)\,z(\xi)\ \ \ \ \ \ \xi \in\,\Oc,\ \ x, y, z \in\,E,\ \ r\in\, \mathbb{R}.\]
Notice that for every  $x_1, y_1, x_2, y_2, z \in\,E$, $r_1, r_2 \in\, \mathbb{R}$
\begin{equation}
\label{bms2}
\vert \sigma(x_1,y_1,r_1)z	-\sigma(x_2,y_2,r_2)z\vert_E\leq L_\sigma\,\left(\vert x_1-x_2\vert_E+\vert y_1-y_2\vert_E+|r_2-r_1|\right)\,\vert z\vert_E.
\end{equation}

\end{enumerate}
}	
\end{Remark}

\subsection{The  operator $A$}

For all $x \in C^2(\Oc)$ we define 
\[
\mathcal{A}x(\xi) := \frac{1}{2}\sum_{i,j=1}^d \frac{\partial}{\partial \xi_i}\left(a_{ij}\, \frac{\partial x}{\partial \xi_j}\right)(\xi) \ \ \ \ \ \ \  \xi \in \bar{\mathcal{O}},
\]
and we assume the following conditions.
\begin{Hypothesis}
	\label{H2} The coefficients $a_{i j}:\Oc\to\mathbb{R}$ are of class $\mathcal{C}^1(\Oc)$. Moreover, the matrix $[a_{ij}(\xi)]_{j,j}$ is  symmetric, for every $\xi \in \Oc$, and strictly positive, uniformly with respect to $\xi \in\,\Oc$. 
\end{Hypothesis}

In what follows, we denote by $A$ the realization in $H$ of  $\mathcal{A}$, endowed with the Dirichlet boundary conditions. Namely,
\[
D(A) = \{ x \in W^{2,2}(\mathcal{O}; \mathbb{R}^r) : x = 0 \text{ in } \partial \mathcal{O} \}, \quad Ax = \mathcal{A}x,\ \ \ \ \ \ x \in\,D(A).
\]
We know that $A $ generates an analytic  and strongly continuous semigroup  $S(s)$, $s\geq 0$, on $H$. Moreover there exists a complete orthonormal system $(e_k)_{k \in\,\mathbb{N}}$ of $H$ and a strictly positive sequence $(\alpha_k)_{k \in\,\mathbb{N}}$ such that 
\[Ae_k=-\alpha_k e_k,\ \ \ \ \ k \in\,\mathbb{N}.\]

\begin{Remark}
	{\em 
	\begin{enumerate}
\item[1.] The semigroup $S(s)$ acts  on $E$ and
\[|S(s)x|_{E} \leq |x|_E.\]
 \item[2.] For every $p\geq 1$ and $\eta\geq 0$, the semigroup $S(s)$ maps $L^p(\mathcal{O})$ into $W^{\eta,p}(\mathcal{O})$ with \begin{equation}\label{stima-reg-semigruppo}
    |S(s)x|_{W^{\eta,p}(\mathcal{O})} \leq c\, (s\wedge 1)^{-\eta/2} |x|_{L^{p}(\mathcal{O})},\ \ \ \ \ \ \ s>0,\ \ \ \ \ \ x \in\,L^p(\mathcal{O}),
    \end{equation}
and the constant  $c$ does not depend on $p$.
\item[3.] There exists a continuous kernel $K: (0,+\infty) \times \R^d \times \R^d \to [0,+\infty)$ such that for all $x\in E$ and $s\geq 0$
\[ S(s)x(\xi)=\int_{\mathcal{O}}K(s,\xi,y)x(y) dy.\]
It is well know that the kernel $K$ satisfies the following bounds
\begin{equation}\label{bms3}0\leq K(s,\xi,y)\leq k_1 s^{-d/2},\ \ \ \ \ \ \ \ \int_{\mathcal{O}}  K(s,\xi,y)dy\leq 1,\end{equation}
and 
\begin{equation}
\label{bms7}	
|\nabla_{\xi} K(s,\xi,y)|\leq k_2 s^{-(d+1)/2},\ \ \ \ \ \ |\partial_s K(s,\xi,y)| \leq k_3 s^{-(d+2)/2}.
\end{equation}

     \end{enumerate}
 }
\end{Remark}
    
    \subsection{The noise}\label{hypNoise}
     
  The Gaussian noise $W_s$, $s \geq 0$,  is given by 
\begin{equation}
\label{bms1}	
W_s=\sum_{i=1}^{\infty} \sqrt{\la_i}\, e_i\,\beta_i(s),\ \ \ \ \ \  s \geq 0,
\end{equation}
 where  $(\lambda_i)_{i \in\,\mathbb{N}}$ is a sequence of non-negative real numbers, $(e_i)_{i \in\,\mathbb{N}}$ is the orthonormal system that diagonalizes $A$  and $(\beta_i)_{i\in \mathbb{N}}$ is a sequence of independent standard (real valued) Brownian motions, all defined on the stochastic basis $(\Omega, \mathcal{F}, (\mathcal{F}_{s})_{s\geq 0}, \mathbb{P})$.  We assume the following conditions.
\begin{Hypothesis}
\label{H3}
If $d=1$, then
\[ \sup_{i\in \mathbb{N}}\la_i<\infty.\]
If $d\geq 2$, then there exists $\theta<d/(d-2)$ such that
\[\sum_{i=1}^\infty \la_i^{\theta}\, |e_i|^2_{L^\infty(\mathcal{O})}<\infty.\]
\end{Hypothesis}
In what follows, we shall denote by $H_0$ the reproducing kernel of the noise $W_s$. $H_0$ is a Hilbert space, endowed with the norm
\[\vert x\vert_{H_0}^2=\sum_{i=1}^\infty \lambda_i\vert\langle x,e_i\rangle_H\vert^2.\]
Finally, we  fix $T>0$ and, for every $p\geq 1$ and $0\leq \lambda < t\leq T$, we  denote by $\mathcal{H}^{\,p}_{\lambda,t}(E)$    the space of adapted $E$-valued processes $X(s)$, $s \in\,[\lambda,t]$, such that
\[\vert X\vert_{ \mathcal{H}^{\,p}_{\lambda,t}(E)}^p:=\mathbb{E}\sup_{s \in\,[\lambda,t]}\vert X(s)\vert_E^p<\infty.\]

\section{Conditional Expectation in separable Banach  spaces}
\label{sec3}
Let $K$ be a separable Banach space.  Given a probability space $(\Omega,\mathcal{F},\mathbb{P})$ and $p\geq 1$, we define
$$L^p(\Omega,\mathcal{F},\mathbb{P};K):=\left\{Z: \Omega \rightarrow K \hbox{ Bockner-measurable such that }\mathbb{E}|X|_K^p<\infty \right\}.$$
Moreover  we denote by $\mathbb{S}$ the vector spaces of simple functions 
\[Z=\sum_{i=1}^n z_i\, I_{A_i},\]
for arbitrary $ n\in \N$, $ z_i\in K$ and $A_i\in \mathcal{F}$.
As well known,  the set  $\S$ is dense in $L^p(\Omega,\mathcal{F},\mathbb{P};K)$.

Now, given another $\sigma$-algebra 
$\mathcal{G}\subseteq \mathcal{F}$ and a simple random variable $Z=\sum_{i=1}^n z_i\, I_{A_i} \in \S $ we  define 
$$ \E(Z|\mathcal{G}):= \sum_{i=1}^nz_i\, \E(I_{A_i}|\mathcal{G}).$$
Notice that this definition does not depend on the particular representation of $Z$.
\begin{Lemma} \label{lemma2.1} The linear functional \begin{equation} \label{bms24}Z\in\,\S\mapsto \E(Z|\mathcal{G}) \in\,L^p(\Omega,\mathcal{G},\mathbb{P};K),\end{equation}
is well defined. 
Moreover 
\begin{equation}\label{bms25}\E\,|\E(Z|\mathcal{G})|_K^p\leq \E\,|Z|_K^p,\ \ \ \ \ p\geq 1.\end{equation}
\end{Lemma}
\begin{proof} The fact that the  linear mapping \eqref{bms24} is well-defined is straightforward. In order to prove  estimate \eqref{bms25}, we first  notice that we can always assume that $Z=\sum_{i=1}^n z_i I_{A_i}  $  for some $z_i\in K$ and $A_i\in \mathcal{F}$, with   $A_i\cap A_j=\emptyset $, if $i\neq j$. Thus, we have
$$\begin{aligned}
\E\,|\E(Z|\mathcal{G})|_K^p  =\E\,\big|\sum_{i=1}^n z_i\, \E(I_{A_i}|\mathcal{G})\big|_K^p &\leq \E\,\left (\sum_{i=1}^n |z_i|_K \E(I_{A_i}|\mathcal{G})\right)^p \\[10pt] = \E\,\left(\E\left(\sum_{i=1}^n |z_i|_KI_{A_i}\big|\mathcal{G}\right)\right)^p  & \leq
\E\left(\E\left(\sum_{i=1}^n |z_i|_K^p I_{A_i}\big|\mathcal{G}\right)\right)=\sum_{i=1}^n |z_i|_K^p\,\P(A_i)=\E\,|Z|_K^p.    
\end{aligned}
$$
and the claim follows.\end{proof}

\medskip
As a consequence of Lemma \ref{lemma2.1},   the linear mapping \eqref{bms24} can be extended to a continuous linear operator, with operatorial norm 1,  from $L^p(\Omega,\mathcal{F},\mathbb{P};K)$ into $L^p(\Omega,\mathcal{G},\mathbb{P};K)$. We will still denote such operator by $\E(Z|\mathcal{G})$.

\smallskip

The next result shows how the conditional expectation we just defined in Lemma \ref{lemma2.1} is in fact a {\em local} operator, in case $K$ is the space $E$ of continuous functions.

\begin{Proposition}  \label{prop-cond-exp-in-C}
If $Z\in  L^p(\Omega, \mathcal{F},\mathbb{P}; E)$ then for every $\xi \in\,\Oc$
\[\mathbb{E}(Z\big | \mathcal{G})(\xi)=\mathbb{E}(Z(\xi)\big | \mathcal{G}), \quad \mathbb{P}-a.s.\]
\end{Proposition}
\begin{proof} If $Z=\sum_{i=1}^n z_i I_{A_i}$,  for some   $n\in \N$, $z_i\in E$, and $A_i\in \mathcal{F}$,  then the claim is straightforward. 
Now, let us fix an arbitrary $Z\in  L^p(\Omega,\mathcal{F},\mathbb{P};E)$ and let $(Z_n)_{n\in \N}$ be a sequence of simple functions such that \[\lim_{n\to\infty} \mathbb{E}|Z_n-Z|_E^p=0.\] Due to the continuity of conditional expectation proved in Lemma \ref{lemma2.1}, we have 
\[\lim_{n\to\infty}\E\,|\E(Z_n|\mathcal{G})-\E(Z|\mathcal{G})|^p_E=0,\] so that, for any fixed $\xi\in \Oc$, we obtain 
\begin{equation}\label{bms35}\lim_{n\to\infty}\E\,|\E(Z_n(\xi)|\mathcal{G})-\E(Z|\mathcal{G})(\xi)|^p=\lim_{n\to\infty}\E\,|(\E(Z_n|\mathcal{G})-\E(Z|\mathcal{G}))(\xi)|^p=0.\end{equation} 
On the other hand, since   
\[\lim_{n\to\infty}E|Z_n(\xi)-Z(\xi)|^p=0,\]
for every fixed $\xi\in\,\Oc$, from the continuity property of standard conditional expectation it follows  
\[\lim_{n\to\infty}\E\,|\E(Z_n(\xi)|\mathcal{G})-\E(Z(\xi)|\mathcal{G})|^p=0.\] 
Thanks to \eqref{bms35}, this allows to conclude that $\E(Z(\xi)|\mathcal{G})= \E(Z|\mathcal{G})(\xi)$.\end{proof}

\begin{Remark}\em It is immediate to check that if $Z=(Z_1,Z_2)\in  L^p(\Omega, \mathcal{F},\mathbb{P}; E\times \mathbb{R})$ then 
\[\mathbb{E}(Z\big | \mathcal{G})=(\mathbb{E}(Z_1\big | \mathcal{G}),\mathbb{E}(Z_2\big | \mathcal{G})).\]
\end{Remark}

\section{Main Results}\label{main}
We fix  $0\leq \lambda<t<T$ and  $x \in\,E$.  With the notations we have introduced in the previous two sections,  for every $\e >0$ equation \eqref{SPDE-intro} can be written as
\begin{equation}
\label{stoch-pde.existence}
\left\{\begin{array}{l}
\ds{dX(s)=\left(A	X(s)+F(X(s))\right)\,ds+\sqrt{\e}\,\sigma(X(s), \E(g(X(t))|\Fs)\,dW_s,}\\[10pt]
\ds{X(\lambda)=x,}
	\end{array}\right.
\end{equation}
for every  $\lambda \leq s \leq t$. 
\begin{Definition}
{\em A process $X \in\,\mathcal{H}^{\,p}_{\lambda,t}(E)$ is a mild solution for equation \eqref{stoch-pde.existence} if 
\begin{equation}\label{mildmain}
X(s)=S(s-\lambda)x+\int_\lambda^s S(s-r)F(X(r))dr
+\sqrt{\e}\int_\lambda^s S(s-r) \sigma(X(r),  \E(g(X(t))|\F_r))\,dW_r,
\end{equation}
holds for every $\lambda\leq s\leq t$.
}
\end{Definition}

Our first result concerns with the study of the well-posedness of equation \eqref{stoch-pde.existence}. 

\begin{Theorem}\label{main-ex-un} 
Under Hypotheses \ref{H1}, \ref{H2} and \ref{H3}, for all $T>0$ and $p\geq 1$ we can find $\e_{T,p}>0$ such that for all  $x\in E$  and $0\leq \lambda<t \leq T$ and all $\e\leq \e_{T,p}$ 
there exists a unique mild solution   $X_{\e,\lambda}^{t,x} \in \mathcal{H}^{\,p}_{\lambda,t}(E)$ for equation \eqref{stoch-pde.existence}.
Moreover, for every $x_1, x_2 \in\,E$, we have
\begin{equation}
\label{bms6}
\mathbb{E}\sup_{s \in\,[\lambda,t]}\left|X^{t,x_1}_{\epsilon,\lambda}(s)-X^{t,x_2}_{\epsilon,\lambda}(s)\right|_E^p\leq c_{T,p}\,|x_1-x_2|_E^p,\ \ \ \ \ \ \epsilon\leq \epsilon_{T,p}.	
\end{equation}

\end{Theorem}

Once proved the well-posedness of equation \eqref{stoch-pde.existence}, we  study the validity of a large deviation principle for the family $\{\mathcal{L}(X_{\e}^{t,x})\}_{\epsilon \in\,(0,\e_{T})}$  in the space $C([0,t];E)$, for some $\e_T>0$, where  \[X_{\e}^{t,x}:=X_{\e,0}^{t,x}.\]

To this purpose, for every $\varphi \in\,L^2(0,t;H_0)$ and $x \in\,E$, we consider the  deterministic controlled problem 
\begin{equation}
\label{sbm2-bis-main}
\frac {dX}{ds}(s)=A	X(s)+F(X(s))+\sigma(X(s),  v_{t}(s,X(s)))\varphi(s),\ \ \ \ X(0)=x,	\end{equation}
where
\[ v_{t}(\rho,y) =  g(Y_\rho^{y}(t)),\ \ \ \ \ \ \rho \in\,[0,t],\ \ \ \ y \in\,E,\]
 and $Y^{y}_\rho$ is the solution of the equation
\[
Y(s)=\int_\rho^s S(s-r)F(Y(r))dr
+ S(s-\rho)y, \ \ \ \ \ s\in\,[\rho,t].
\]
We will see that equation \eqref{sbm2-bis-main} admits a unique mild solution  $X^{t,x,\varphi} \in\,C([0,t];E)$. In particular, we can state the following result.

\begin{Theorem}
\label{LDP} Under Hypotheses \ref{H1}, \ref{H2} and \ref{H3}, for every $T>0$ there exists some $\e_T>0$ such that, for every  $0<t\leq T$ and $x \in\, E$, the family $\{\mathcal{L}(X_{\e}^{t,x})\}_{\epsilon \in\,(0,\e_{T})}$ satisfies a large deviation principle in $C([0,t];E)$, with speed $\e$ and  action functional 	
\[I_{t,x}(X)=\frac 12\,\inf\left\{\int_0^t\vert \varphi(s)\vert_{H_0}^2\,ds\,:\, X(s)=	X^{t,x,\varphi}(s),\ s \in\,[0,t]\right\},\]
where $X^{t,x,\varphi}$ is the unique mild solution of equation \eqref{sbm2-bis-main}.
\end{Theorem}

\section{The stochastic convolution}

In this section, we want to study the equation 
\begin{equation}
\label{bms45}
	dZ(s)=	A Z(s)\,ds+\sigma(Z(s),\mathbb{E}(g(Z(t))|\mathcal{F}_s))\,dW_s,\ \ \ \ Z(\lambda)=0,
\end{equation}
for every fixed $0\leq \lambda<t\leq T$ and $s \in\,[\lambda,t]$.
A process $Z^t \in\,\mathcal{H}^{\,p}_{\lambda,t}(E)$ is a mild solution of \eqref{bms45} if
\begin{equation}\label{sbm2}Z^t(s)=\int_\lambda^s S(s-r) \sigma(Z^t(r),  \E(g(Z^t(t))|\F_r))\,dW_r=:\gamma_{\lambda,t}(Z^t)(s),\ \ \ \ \ \ s \in\,[\lambda,t].\end{equation}

\begin{Proposition}\label{prop stoch conv}
Under Hypotheses \ref{H1}, \ref{H2} and \ref{H3}, there exists $\bar{p}\geq 1$ such that $\gamma_{\lambda, t}$ maps $ \mathcal{H}^{\,p}_{\lambda,t}(E)$ into itself, for every $p\geq \bar{p}$, and is Lipschitz-continuous, with Lipschitz constant $
{L_{\gamma_{\lambda,t}}(p)}$. Moreover
\[\sup_{0\leq \lambda<t\leq T}
{L_{\gamma_{\lambda,t}}(p)}=:
{ L_{\gamma}(T,p)} <\infty.\] 
\end{Proposition}
\begin{proof}  If we fix  $X_1, X_2 \in \mathcal{H}^{\,p}_{\lambda,t}(E)$ and define $\delta\gamma_{\la, t}:=\gamma_{\la, t}(X_2)-\gamma_{\la, t}(X_1)$,    for every  $\alpha\in (0,1/2)$ we have
$$ \delta\gamma_{\la, t}(s) =\dfrac{\sin \pi \alpha}{\pi} \int_\lambda^s (s-r)^{\alpha-1}S(s-r)V^{\alpha}_t(r)dr,\ \ \ \ \ \ \ s \in\,[\lambda,t],$$ 
where
\[\displaystyle V^{\alpha}_t(r):= \int_\lambda^r (r-\rho)^{-\alpha}S(r-\rho)Y_t(\rho)dW_{\rho},\]
 and 
\begin{equation}\label{def-deltaG}
  Y_t(\rho) :=  \sigma(X_2(\rho), \E(g(X_2(t))|\F_{\rho}))-\sigma(X_1(\rho), \E(g(X_1(t))|\F_{\rho})).  
\end{equation}
We notice that, according to Proposition \ref{prop-cond-exp-in-C}, for all $h\in E$ we have 
$$ [Y_t(\rho)h](\xi) = Y_t(\rho)(\xi)\,h(\xi),\ \ \ \ \ \ \xi \in\,\Oc.$$
In view of  \eqref{stima-reg-semigruppo}, we have
$$ | \delta\gamma_{\la, t}(s)|_{W^{\eta,p}(\mathcal{O})}\leq c_{_\alpha }\int_\lambda^s (s-r)^{\alpha-1-\eta/2}|V^{\alpha}_t(r)|_{L^p(\mathcal{O})}dr.$$
We choose  $\bar{p}$ large enough so that $  \frac{1}{\bar{p}}<\alpha<\frac{1}{2}$ and $0<\eta<2\alpha -\frac{2}{\bar{p}}$. Then  $(\alpha-1-\frac{\eta}{2})\frac{p}{p-1}>-1$, for every $p\geq \bar{p}$, and  
\[\begin{array}{l}
\ds{	| \delta\gamma_{\la, t}(s)|_{W^{\eta,p}(\mathcal{O})}\leq c_{_\alpha }\left(\int_\lambda^s r^{(\alpha-1-\frac{\eta}{2})\frac{p-1}{p}}dr \right)^{\frac{p}{p-1}}\left(\int_\lambda^s |V^{\alpha}_t(r)|^p_{L^p(\mathcal{O})}dr\right)^{\frac{1}{p}}}\\[14pt]
\ds{\quad \quad \quad \quad \quad \quad \quad \quad \quad \quad  \leq c_{\alpha, p,t}\,\left(\int_\lambda^s |V^{\alpha}_t(r)|^p_{L^p(\mathcal{O})}dr\right)^{\frac{1}{p}}.}\end{array}\]
Notice that $c_{\alpha, p,t}\leq c_{\alpha, p}^T$ for all $t\in (0,T]$.

Next, if we also assume that  $\bar{p}>d+2$, $  \frac{d+2}{2\bar{p}}<\alpha<\frac{1}{2}$ and $\frac{d}{\bar{p}}<\eta<2\alpha -\frac{2}{\bar{p}}$, taking into account that in this case $W^{\eta,\bar{p}}(\mathcal{O})$ is  embedded into $E$, for every $p\geq \bar{p}$ we get
\begin{equation}\label{stimagammav_alpha}
    \sup_{s\in [\lambda,t]}| \delta\gamma_{\la, t}(s)|^p_E \leq c_{\alpha,p,t} \int_\lambda^t |V^{\alpha}_t(r)|^p_{L^p(\mathcal{O})}dr.\end{equation}
    
Now, for every $r\in [\lambda,t]$ and $\xi\in \Oc$, we have
$$|V^{\alpha}_t(r,\xi)|^p= \left|\sum_{i=1}^{\infty} \sqrt{\la_i}\int_\lambda^r (r-\rho)^{-\alpha}S(r-\rho)(Y_t(\rho)e_i) (\xi)d\beta_{i}(\rho) \right|^p,\quad $$
and from the Burkholder-Davis-Gundy inequality, the independence of the Brownian motions   \((\beta_i)_{i\in \mathbb{N}}\) and Hypothesis \ref{H3}, we get
\begin{align*}
&\mathbb{E}|V^{\alpha}_t(r,\xi)|^p \leq c  \, \mathbb{E}\left(\sum_{i=1}^{\infty} \la_i\int_\lambda^r (r-\rho)^{-2\alpha}\left|S(r-\rho)(Y_t(\rho)e_i) (\xi)\right|^2d\rho \right)^{p/2} \\[10pt]
&\quad \quad =c  \, \mathbb{E}\left(\sum_{i=1}^{\infty} \la_i |e_i|_{L^{\infty}(\mathcal{O})}^ {2/\theta}\int_\lambda^r (r-\rho)^{-2\alpha}\left|S(r-\rho)(Y_t(\rho)e_i) (\xi)\right|^2 |e_i|_{L^{\infty}(\mathcal{O})}^ {-2/\theta}d\rho \right)^{p/2}.
\end{align*}
Hence, we get
\begin{align}\label{stima v 1}\nonumber & \mathbb{E}|V^{\alpha}_t(r,\xi)|^p\\[10pt]
&\leq c  \,  \mathbb{E}\left(\int_\lambda^r (r-\rho)^{-2\alpha} \left(\sum_{i=1}^{\infty} \la_i^{\theta} |e_i|_{L^{\infty}(\mathcal{O})}^ {2}\right)^{1/\theta}\left( \sum_{i=1}^{\infty}|S(r-\rho)(Y_t(\rho)e_i) (\xi)|^{ 2\zeta} |e_i|_{L^{\infty}(\mathcal{O})}^ {-2 (\zeta-1)}\right)^{1/\zeta}  d\rho\right)^{p/2}\\[10pt]
\nonumber &\quad \quad \quad \quad  \leq c  \,  \mathbb{E}\left(\int_\lambda^r (r-\rho)^{-2\alpha} A^{1/\zeta}(r,\rho,\xi)  d\rho\right)^{p/2},
\end{align}
where $ \frac{1}{\theta} + \frac{1}{\zeta}=1$ and
\begin{equation}\label{def-A}
    A(r,\rho,\xi):=\sum_{i=1}^{\infty}|S(r-\rho)(Y_t(\rho)e_i) (\xi)|^{ 2\zeta} |e_i|_{L^{\infty}(\mathcal{O})}^ {-2 (\zeta-1)}.
\end{equation}

We have
\begin{align}\label{stimaSG1}
 & \nonumber|S(r-\rho)(Y_t(\rho)e_i) (\xi)|^2 = \Big( \int_{\O} K(r-\rho,\xi,y) Y_t(\rho,y)e_i(y) \,dy \Big)^2 \\[10pt]&\quad \quad\quad   \leq |e_i|^2_{L^{\infty}(\mathcal{O})}   \left(  \int_{\O} K(r-\rho,\xi,y)|Y_t(\rho,y)| \,dy \right)^2   =  |e_i|^2_{L^{\infty}(\mathcal{O})}  \big(S(r-\rho)|Y_t(\rho)|(\xi)\big)^2.
\end{align}
On the other hand, we have
\begin{align*}
&\sum_{i=1}^{\infty} |S(r-\rho)(Y_t(\rho)e_i)(\xi)| ^2 =   
\sum_{i=1}^{\infty} \langle  K(r-\rho,\xi,\cdot)Y_t(\rho),e_i \rangle ^2_{H}= | K(r-\rho,\xi,\cdot)Y_t(\rho)|^2_{H}  \\[10pt]&\quad \quad \quad = \int_\O 
K^2(r-\rho,\xi,y)Y_t ^2(\rho,y) \, dy,
\end{align*}
so that, in view of \eqref{bms3}, we deduce 
\begin{align}\label{stimaSG2}
  \nonumber &\sum_{i=1}^{\infty} |S(r-\rho)(Y_t(\rho)e_i)(\xi)| ^2  \leq c\, (r-\rho)^{-\frac{d}{2}}\int_\O 
K(r-\rho,\xi,y)Y_t^2(\rho,y) \, dy \\[9pt]&\quad \quad \quad \quad \quad \quad = c\, (r-\rho)^{-\frac{d}{2}}\,S(r-\rho)Y_t^2(\rho)(\xi).
\end{align}
Therefore, if we combine together 
\eqref{stimaSG1} and \eqref{stimaSG2}, we get 
\begin{align}\label{bms17}
&\nonumber A(r,\rho,\xi)\leq \left|S(r-\rho)|Y_t(\rho)|(\xi)\right|^{2(\zeta-1)} \sum_{i=1}^{\infty} |S(r-\rho)(Y_t(\rho)e_i )(\xi)| ^2 \\[10pt]
&\leq c \big(S(r-\rho)|Y_t(\rho)|(\xi)\big)^{2(\zeta-1)} (r-\rho)^{-\frac{d}{2}}\,S(r-\rho)Y_t^2(\rho)(\xi)\leq c\,  (r-\rho)^{-\frac{d}{2}}|Y_t(\rho)|_E^{2\zeta}.
	\end{align}
 Plugging the above inequality into \eqref{stima v 1}, we obtain
$$ \mathbb{E}|V^{\alpha}_t(r,\xi)|^p \leq c\, \mathbb{E}\left(\int_\lambda^r (r-\rho)^{-2\alpha-\frac{d}{2\zeta}} |Y_t(\rho)|_E^{2}\,d\rho\right)^{p/2},$$
and the Young inequality gives
\begin{align}\label{stimaconvalpha}
      &\nonumber\mathbb{E}\int_\lambda^s|V^{\alpha}_t(r)|^p_{L^p(\mathcal{O})}\,dr = \mathbb{E}\int_\lambda^s\int_{\mathcal{O}}|V^{\alpha}_t(r,\xi)|^p\,d\xi\,dr\\[10pt]      & \quad \quad \leq c  \, \left(\int_\lambda^s r^{-2\alpha-\frac{d}{2\zeta}}\,dr \right)  \int_\lambda^s \mathbb{E}|Y_t(r)|^p_E\,dr\leq c\,\int_\lambda^s \mathbb{E}|Y_t(r)|^p_E\,dr.
 \end{align}

Now, we claim that
\begin{equation}  \label{stima delta sigma} 
\mathbb{E}|Y_t(r)|^p_E  \leq c_t\, \mathbb{E}\, \sup_{\lambda\leq s \leq t}|X_1(s)- X_2(s)|^p_E.\end{equation} 
Actually, according to Hypothesis \ref{H1}, we have
\begin{align*}
  &|Y_t(r)|_E=\left|\sigma(X_1(r), \mathbb{E}(g(X_1(t))|\mathcal{F}_{r}))- \sigma(X_2(r),  \mathbb{E}(g(X_2(t))|\mathcal{F}_{r}))\right|_E \\[10pt] &
   \quad \quad \quad \leq L_\sigma \Big( |X_1(r)- X_2(r)|_E   +|\mathbb{E}\left(g(X_1(t)) - g(X_2(t)\right) |\mathcal{F}_{r})|_E \Big). 
\end{align*}
By taking into account what proved in Section \ref{sec3}, namely that  $\mathbb{E} (\,\cdot\, |\mathcal{F}_s) $ acts as a continuous linear operator, with operatorial norm 1,  from $L^p(\Omega,\mathcal{F}_t,\mathbb{P};E)$ into $L^p(\Omega,\mathcal{F}_\rho,\mathbb{P};E)$,  for all $\rho\in [0,t]$, we have
\[\mathbb{E}\,|\mathbb{E}\left(g(X_1(t)) - g(X_2(t)) |\mathcal{F}_{\rho}\right)|_E ^p\leq c\,\mathbb{E}\,|g(X_1(t)) - g(X_2(t))|_E^p
 \leq c\,L^p_g\,\mathbb{E}\,|X_1(t) - X_2(t)|_E^p,\]
 and \eqref{stima delta sigma} follows.

\smallskip

Finally plugging \eqref{stima delta sigma} in \eqref{stimaconvalpha} we deduce that
\begin{equation}
    \mathbb{E}\int_\lambda^s|V^{\alpha}_t(r)|^p_{L^p(\mathcal{O})}\,dr \leq c_t  \, \mathbb{E} \sup_{\lambda\leq s \leq t}|X_1(s)- X_2(s)|^p_E, \qquad s \in [\lambda,t],
 \end{equation}
 and  in view of  \eqref{stimagammav_alpha} we get that there exists some constant ${L_{\gamma_{\lambda,t}}(p)}>0$, which is increasing with respect to $t \in\,[\lambda,T]$, such that 
\begin{equation}
     \mathbb{E} \sup_{\lambda\leq s \leq t}|\gamma_{\la, t}(X_1)(s)- \gamma_{\la, t}(X_2)(s)|^p_E \leq 
   {L_{\gamma_{\lambda,t}}(p)}\, \mathbb{E} \sup_{\lambda\leq s \leq t}|X_1(s)- X_2(s)|^p_E.
 \end{equation}
\end{proof}

\section{Salin's solution map}\label{sec5}

In this section we adapt to the case of bounded domains with Dirichlet boundary conditions an argument introduced by Salins in \cite[Section 5]{Salins}. Such argument allows to prove the existence and uniqueness of mild solutions for reaction-diffusion equations with reaction terms satisfying Hypothesis \ref{H1}-2.
 
For every $T>0$ and   $z\in C([\lambda,t];E)$, with $0\leq \lambda<t\leq T$,  we consider the equation 
\begin{equation}
    \label{SalinsMild}
    u(s,\xi)=\int_\lambda^s S(s-r)F(u(r))(\xi)dr+z(s,\xi),
\end{equation}
for $s \in\,[\lambda,t]$ and $\xi \in\,\Oc$.

\begin{Theorem}\label{PropSalins} Under Hypotheses \ref{H1} and \ref{H2}, for any $T>0$ and  $z\in C([\lambda,t];E)$, with $0\leq \lambda<t\leq T$, 
there exists a unique solution $u\in C([\lambda,t];E)$ to equation  \eqref{SalinsMild}.

Moreover, if we denote such  solution  by $\mathcal{M}_{\lambda,t}(z)$, we have that  $\mathcal{M}_{\lambda, t}$  is a Lipschitz map  from $ C([\lambda,t];E)$ into itself. Finally, if we  denote by $L_{\mathcal{M}_{\lambda,t}}$ the Lipschitz constant of $\mathcal{M}_{\lambda, t}$  in $C([\lambda,t];E)$, we have  
\begin{equation}
	\label{bms13}
	\sup_{0\leq \lambda<t\leq T}L_{\mathcal{M}_{\lambda,t}}=:L_{\mathcal{M}}(T)<\infty.
\end{equation}
\end{Theorem}
Before proving Theorem \ref{PropSalins}, we need some preliminary results.
\begin{Lemma}\label{Salins a priori}  If  $u, z\in C([\lambda,t];E)$  verify \eqref{SalinsMild}, for all $s\in [\lambda,t]$ and $\xi \in \Oc$, then
\begin{equation}\label{bms10}
    \sup_{s \in\,[\lambda,t]}|u(s)|_E\leq c_T \left( \sup_{s\in [\lambda,T]} |F(z(s))|_E+ \sup_{s\in [\lambda,t]}|z(s)|_E\right),
\end{equation}
for some constant $c_T$ which depends only on $T$, $f$ and $K$.
\end{Lemma}
\begin{proof} If we define $v:=u-z$, we have 
$$
    v(s,\xi)=\int_\lambda^s S(s-r)F(u(r))(\xi)dr=\int_\lambda^s\int_{\mathcal{O}} K(s-r,\xi, y)f(u(r,y)) dy\,dr.
$$
Since the semigroup $S(s)$ is analytic, we have that $v(s,\cdot)\in C^2(\O)$, for all $s \in\,[\lambda,t]$, and  satisfies the equation
\[\left\{
\begin{array}{l}
\ds{\partial_s v(s,\xi)=\mathcal{A} v(s,\xi)+f(v(s,\xi)),}\\[10pt]
\ds{v(\lambda,\xi)=0,\ \ \xi \in\,\bar{\mathcal{O}},\ \ \ \ \ v(s,\xi)=0,\ \ \xi \in\,\partial\mathcal{O}.}	
\end{array}\right.\]
 In particular, we have
\begin{equation}\label{sbm1-bis}
    \partial_s\, v^2(s,\xi)=v(s,\xi)\sum_{i,j=1}^d \frac{\partial}{\partial \xi_i}\left(a_{ij}\, \frac{\partial v}{\partial \xi_j}(s,\cdot)\right)(\xi)+\rho(s,\xi),
\end{equation}
where
\begin{align} \label{sbm3}
    \rho(s,\xi):= 2 v(s,\xi)f(u(s,\xi))= \,& 2 v(s,\xi)\left[f(z(s,\xi)+v(s,\xi))-f(z(s,\xi))\right]+  2 v(s,\xi)f(z(s,\xi)) \nonumber \\[10pt]  \quad \quad \quad  \leq  & (2k+1)v^2(s,\xi)+ f^2(z(s,\xi)).
\end{align}
We observe that 
\begin{align*}
  & v(s,\xi)\sum_{i,j=1}^d \frac{\partial}{\partial \xi_i}\left(a_{ij}  \frac{\partial v}{\partial \xi_j}(s,\cdot)\right)(\xi) =\\
  &\quad \quad \quad \quad \quad \quad  = \frac{1}{2} \sum_{i,j=1}^d \frac{\partial}{\partial \xi_i}\left(a_{ij} \frac{\partial v^2}{\partial \xi_j}(s,\cdot)\right)(\xi) - \sum_{i,j=1}^d a_{ij} (\xi) \frac{\partial v}{\partial \xi_i}(s,\xi) \frac{\partial v}{\partial \xi_j}(s,\xi)
  \leq  \mathcal{A}v^2(s,\xi),
\end{align*} 
so that, from \eqref{sbm1-bis} and \eqref{sbm3} we obtain
\[\partial_s v^2(s,\xi)\leq \mathcal{A}v^2(s,\xi)+(2k+1)v^2(s,\xi)+f^2(z(s,\xi)).\]
Thus if we define $\vartheta(s,\xi):=\exp(-(2k+1)s)v^2(s,\xi)$, we get
$$\begin{cases}
     \partial_s \vartheta(s,\xi)\leq A\, \vartheta(s,\xi)+e^{-(2k+1)s}f^2(z(s,\xi)),\ \ \ \ \ s\in [\lambda,t],\\[10pt]
     \vartheta(s,\xi)=0, \ \ \  \xi\in\partial \mathcal{O},\ \ \  s\in [\lambda,t],\ \ \ \ \ \ \ 
     \vartheta(0,\xi)=0 , \ \ \ \  \xi\in\,\Oc,
     \end{cases}$$
     and this implies
     $$ \vartheta(s,\xi)\leq \int_\lambda^se^{-(2k+1)r}\int_{\mathcal{O}} K(s-r,\xi, y)f^2(z(r,y))dy\,dr.$$
 Since  $\int_{\mathcal{O}}  K(s,\xi,y)dy \leq 1$, this gives
    \[\vartheta(s,\xi)\leq c\; \sup_{r\in [\lambda,t],\, \xi \in \Oc} f^2(z(r,\xi)).\] 
    In particular, we get
    \[\sup_{s \in\,[\lambda,t]}\vert v(s)\vert_E\leq c_T\,\sup_{s \in\,[\lambda,t]}\vert f(z(s))\vert_E,\]
    and \eqref{bms10} follows.  \end{proof}

\begin{Lemma} For every  $z\in C([\lambda,t];E)$ 
there exists a solution $u\in C([\lambda,t];E)$ to equation  \eqref{SalinsMild}. 
\end{Lemma}
\begin{proof}
If we set  $\varphi(u):=f(u)- ku$, there exists some $n_0 \in\,\mathbb{N}$ such that for every $n\geq n_0$ and $u \in\,\mathbb{R}$  we can define 
\[\varphi_n(u):= n\left(\left(I-\varphi/n\right)^{-1}(u)-u\right),\ \ \ \ \ \  f_n(u)=  \varphi_n(u) + ku.\] As shown for example in  \cite[Proposition D.11]{DpZ1} we have
\begin{equation}
  \label{dissipatività unif in n}
|f_n(u)|\leq |f(u)|+ |k||u|,\ \ \ u\in \mathbb{R}.   
\end{equation} 
 For every $n\geq n_0$ the function $f_n:\mathbb{R}\to \mathbb{R}$ is Lipschitz-continuous and  
 \[f_n(u_2)-f_n(u_1) \leq k(u_2-u_1),\ \ \ \ \ \ u_1<u_2.\]
 Moreover, for every $R>0$
 \begin{equation}
 \label{bms8}
 \lim_{n\to\infty}\,\sup_{|u|\leq R}|f_n(u)-f(u)|=0.	
 \end{equation} 
Being $f_n$ Lipschitz-continuous, there exists a unique $u_n\in C([\lambda,t];E)$ such that 
\begin{equation}
    \label{SalinsMild_n}
    u_n(s,\xi)=\int_\lambda^s\int_{\mathcal{O}} K(s-r,\xi, y)f_n(u_n(r,y)) dy\,dr+z(s,\xi),\ \ \ \ \ s \in\,[\lambda,t],\ \ \xi \in\,\Oc.
\end{equation}
Moreover, thanks to  Lemma \ref{Salins a priori},  for all $s \in\,[\lambda,t]$ and $n\in \mathbb{N}$
\begin{equation}
    |u_n(s)|_E\leq c_T \left( \sup_{s\in [0,T]} |f_n(z(s))|_E+ \sup_{s\in [0,T]}|z(s)|_E\right).
\end{equation}
Thus, in view of     \eqref{ipotisi su f} and \eqref{dissipatività unif in n}
\begin{equation}\label{6.9}
 \sup_{n\in \mathbb{N}} \,\sup_{s\in[\lambda,t]} |u_n(s)|_E <\infty,\ \ \ \ \ \ \ \ \ \ 
 \sup_{n\in \mathbb{N}} \,\sup_{s\in[\lambda,t]} |f_n(u_n(s))|_E <\infty.  
\end{equation}

Now, if we set $v_n:=u_n-z$, we have
\begin{equation}\label{bms39}
    v_n(s,\xi)=\int_\lambda^s\int_{\mathcal{O}} K(s-r,\xi, y)f_n(u_n(r,y)) dy\,dr,\ \ \ \ \ s \in\,[\lambda,t],\ \ \xi \in\,\Oc.
\end{equation}
Consequently for $\xi_i\in \Oc$, $i=1,2$, and $s \in\,[\lambda,t]$
\begin{equation*}
    |v_n(s,\xi_2)-v_n(s,\xi_1)|\leq\int_\lambda^s\int_{\mathcal{O}} | K(s-r,\xi_2, y)-  K(s-r,\xi_1, y)| \,|f_n(u_n(r,y))| dy\,dr.
\end{equation*}
Therefore, , for   any $\eta\in (0,1)$
\begin{align*}
    |v_n(s,\xi_2)-v_n(s,\xi_1)|\leq c\,\int_\lambda^s\left(\int_{\mathcal{O}} ( K(s-r,\xi_2, y) +  K(s-r,\xi_1, y)) dy\right)^{1-\eta}dr \\ \times \int_\lambda^s\left(\int_{\mathcal{O}}  \int_0^1 |\nabla_{\xi} K(s-r, \theta\xi_1+(1-\theta)\xi_2, y)| d\theta\,  dy\right)^{\eta}dr\,  |\xi_2-\xi_1|^\eta, 
\end{align*}
and, thanks to \eqref{bms7}, we conclude
\begin{equation*}
     |v_n(s,\xi_2)-v_n(s,\xi_1)|\leq 
     c\, |\xi_2-\xi_1|^{\eta}\int_\lambda^s (s-r)^{-\eta\,(d+1)/2} dr.
\end{equation*}
In particular, if $\eta< 2/(d+1)$, then $v_n(s)$ is $\eta$-H\"{older}-continuous,  uniformly with respect to $s\in[\lambda,t]$ and $n\geq n_0$.

In a similar way if $\lambda\leq s_1< s_2\leq t$ and $\xi\in\Oc$
\begin{align*}
    &|v_n(s_2,\xi)-v_n(s_1,\xi)|\leq\int_\lambda^{s_1}\int_{\mathcal{O}} | K(s_1-r,\xi, y)-  K(s_2-r,\xi, y)| \,|f_n(u_n(r,y))| dy\,dr \\[10pt]& \quad \quad \quad \quad \quad + 
    \int_{s_1}^{s_2}\int_{\mathcal{O}} K(s_2-r,\xi, y)| \,|f_n(u_n(r,y))|\, dy\,dr.
\end{align*}
Proceeding as above we get
$$\int_\lambda^{s_1}\int_{\mathcal{O}} | K(s_1-r,\xi, y)-  K(s_2-r,\xi, y)| \,|f_n(u_n(r,y))| dy\,dr \leq c (s_2-s_1)^{\alpha}\int_\lambda^{s_1} (s_1-r)^{-\alpha(d+2)/2} dr,
$$
and
$$  \int_{s_1}^{s_2}\int_{\mathcal{O}} K(s_2-r,\xi, y)| \,|f_n(u_n(r,y))| dy\,dr\leq  c\int_{s_1}^{s_2}\int_{\mathcal{O}} K(s_2-r,\xi, y) dy\,dr \leq c(s_2-s_1).$$
This implies that if $\alpha< 2/(d+2)$ the function  $v_n(\cdot,\xi )$ is $\alpha$-H\"{older}-continuous in $s$, uniformly with respect to $\xi \in\,\Oc$ and $n\geq n_0$.
Consequently, by Arzel\`a-Ascoli Theorem  there exists $v \in C([\lambda,t];E)$ with 
\[\lim_{n\to\infty}\sup_{s \in\,[\lambda,t]}\vert v_n(s)- v(s)\vert_E=0. \]

Finally, recalling that $u_n= v_n+z$, thanks to \eqref{bms8} and  \eqref{6.9}, we can pass to the limit in \eqref{bms39} and we obtain

\begin{equation*}
    v(s,\xi)=\int_\lambda^s\int_{\mathcal{O}} K(s-r,\xi, y)f(v(r,y))+z(r,y)) dy\,dr.
\end{equation*}
 Thus, if we set $u:=v+z$, we can conclude that $u$ solves equation \eqref{SalinsMild}. \end{proof}

 \subsection{Proof of Theorem \ref{PropSalins}}
Let $z_i\in C([\lambda,t];E)$, for $i=1, 2$, and let   $v_i:=u_i-z_i$, where $u_i$ is a solution of equation \eqref{SalinsMild} associated with $z_i$. Moreover let $\bar{z}:=z_2-z_1$, $\bar{v}:=v_2-v_1$, and $\bar{f}:=f(v_2+z_2)-f(v_1+z_1)$. 
We have 
\begin{equation*}\label{eq per vn}
    \bar{v}(s,\xi)=\int_\lambda^s\int_{\mathcal{O}} K(s-r,\xi, y)\bar{f}(r,y) dy\,dr, \ \ \ \ \ \ \ s \in\,[\lambda,t],\ \ \ \xi \in\,\Oc,
\end{equation*}
and $\bar{v}(s,\xi)=0$, for $\xi \in \partial \mathcal{O}$, so that for all $s\in[\lambda,t]$ there exists $\bar{\xi}(s) \in\,\mathcal{O}$ such that 
\[\sup_{\xi\in \Oc}|\bar{v}(s,\xi)|=|\bar{v}(s,\bar{\xi}(s))|.\]
In particular,  $\nabla_\xi\bar{v}(s,\bar{\xi}(s))=0 $ and
$$\sign(\bar{v}(s,\bar{\xi}(s)))\,\mathcal{A}\bar{v}(s)(\bar{\xi}(s))=\sign(\bar{v}(s,\bar{\xi}(s)))\textrm{Tr}\left[a(\bar{\xi}(s))D^2_{\xi}\,\bar{v}(s,\bar{\xi}(s))\right]\leq 0,$$ 
where $a(\xi)$ is the matrix of entries $a_{ij}(\xi)$,  $i,j\in 1,...,d$, and $D^2_{\xi}{v}$ is the Hessian matrix of $v$. Indeed, if $\bar{v}(s,\bar{\xi}(s))>0$ then $\bar{\xi}(s)$ is a maximum point and $D^2_{\xi}\,\bar{v}(s,\bar{\xi}(s))$  is negative semidefinite and the converse holds if  $\bar{v}(s,\bar{\xi}(s))<0$. Moreover 
\[\sign(\bar{v}(s,\bar{\xi}(s)))\,\bar{v}(s,\bar{\xi}(s))  =|\bar{v}(s)|_E.\]


By proceeding as in \cite{Salins}, we have
\begin{align*} 
&\frac{d}{ds}^-|\bar{v}(s)|_E =\sign(\bar{v}(s,\bar{\xi}(s)))\,\partial_s\bar{v}(s,\bar{\xi}(s))=\sign(\bar{v}(s,\bar{\xi}(s)))\, \left(\mathcal{A}\bar{v}(s,\bar{\xi}(s))+\bar{f}(s,\bar{\xi}(s))\right) \\[10pt]
& \quad \quad \quad \quad \quad \quad \quad \quad \leq \sign(\bar{v}(s,\bar{\xi}(s))) \bar{f}(s,\bar{\xi}(s))=:\Phi(s).
\end{align*}

Now, if $\sign(\bar{v}(s,\bar{\xi}(s)))=\sign(\bar{v}(s,\bar{\xi}(s))+\bar{z}(s,\bar{\xi}(s)))$ then
\begin{align*} \Phi(s)= &\sign(\bar{v}(s,\bar{\xi}(s))+\bar{z}(s,\bar{\xi}(s)))\Big( {f}(v_2(s,\bar{\xi}(s))+ z_2(s,\bar{\xi}(s)))  -  {f}(v_1(s,\bar{\xi}(s))+ z_1(s,\bar{\xi}(s))) \Big)\\[10pt] &\quad  \leq
 k\sign(\bar{v}(s,\bar{\xi}(s))+\bar{z}(s,\bar{\xi}(s)))  (\bar{v}(s,\bar{\xi}(s))+\bar{z}(s,\bar{\xi}(s)))\leq k\,\vert \bar{v}(s)\vert_E+k\,\vert \bar{z}(s)\vert_E. 
\end{align*}
On the other hand, if $\sign(\bar{v}(s,\bar{\xi}(s))) \neq \sign(\bar{v}(s,\bar{\xi}(s))+\bar{z}(s,\bar{\xi}(s)))$  then $| \bar{z}(s, \bar{\xi}(s))|> | \bar{v}(s, \bar{\xi}(s))|$ and consequently 
$$  | \bar{v}(s)|_E\leq  | \bar{z}(s)|_E.$$ 
Thus, we conclude that 
 \[\Phi(s)\leq k |\bar{v}(s)|_E+  k |\bar{z}(s)|_E+\textrm{I}_{\{\sign(\bar{v}(s,\bar{\xi}(s))) \neq \sign(\bar{v}(s,\bar{\xi}(s))+\bar{z}(s,\bar{\xi}(s)))\}}\Phi(s),\] and consequently
$$ \Phi(s)\leq k |\bar{v}(s)|_E+  k |\bar{z}(s)|_E+\textrm{I}_{\{  | \bar{v}(s)|_E\leq  | \bar{z}(s)|_E \}}\Phi(s).$$
In particular
\begin{equation}\label{bms11}  \frac{d}{ds}^{-}|\bar{v}(s)|_E\leq k|\bar{v}(s)|_E+k\sup_{s\in [\lambda,t]}|\bar{z}(s)|_E + \textrm{I}_{\{ | \bar{v}(s)|_E\leq  | \bar{z}(s)|_E\}}\Phi(s).\end{equation}

Now, in \cite[Section 5]{Salins} it is proved that if  $\varphi $, $\psi$ and $\Phi$ are non-negative functions defined on $[\lambda,t]$, with $\varphi(\lambda)=0$ and 
    $$ \dfrac{d^-}{dt}\varphi(t)\leq c_1\varphi(t)+c_2+\Phi(t)\hbox{\rm{I}}_{\{\varphi(t)\leq \psi(t)\}},$$
    then 
     $$ \sup_{s \in\,[\lambda,t]}\varphi(s)\leq \left(c_2 (t-\lambda)+\sup_{s\in [\lambda,t]}\psi(s)\right)e^{c_1 (t-\lambda)}.$$
Thus, if we apply  this  to \eqref{bms11}, with 
\[\varphi(s)=| \bar{v}(s)|_E,\ \ \psi(s)=| \bar{z}(s)|_E,\ \ c_1=k,\ \ c_2= \sup_{s\in [\lambda,t]}|\bar{z}(s)|_E,\] we get
\begin{equation} \label{bms12}\sup_{s\in [\lambda,t]}|\bar{v}(s)|_E\leq (k(t-\lambda)+1)\sup_{s\in [\lambda,t]}|\bar{z}(s)|_E\,e^{k (t-\lambda)}.\end{equation}
In particular, if $z_1=z_2$, this implies that $v_1=v_2$, so that $u_1=u_2$ and uniqueness holds for equation \eqref{SalinsMild}. 

As mentioned above, this allows to define  $\mathcal{M}_{\lambda,t}(z)$ as the unique solution of equation \eqref{SalinsMild}, associated with $z \in\,C([\lambda,t];E)$. Due to \eqref{bms12}, $\mathcal{M}_{\lambda,t}:C([\lambda,t];E)\to C([\lambda,t];E)$ is Lipschitz-continuous, with Lipschitz constant \[L_{\mathcal{M}_{\lambda,t}}\leq (k(t-\lambda)+1)\,e^{k (t-\lambda)}+1,\]
so that \eqref{bms13} follows.

\begin{Remark}
{\em Due to \eqref{bms10}, we have
\begin{equation}\label{bms21}\sup_{s \in\,[\lambda,t]}|\mathcal{M}_{\lambda,t}(z)(s)|_E\leq c_T \left( \sup_{s\in [\lambda,T]} |F(z(s))|_E+ \sup_{s\in [\lambda,t]}|z(s)|_E\right).\end{equation}
}	
\end{Remark}

\section{Proof of Theorem \ref{main-ex-un}}
\begin{proof}
Equation \eqref{mildmain} can be rewritten as
$$ X(s)=\mathcal{M}_{\lambda,t}(\sqrt{\e}\,\gamma_{\la, t}(X)+S(\cdot-\lambda)x)(s),\ \ \ \ \ \ \ s \in\,[\lambda,t],$$
where $\mathcal{M}_{\lambda, t}:C([\lambda,T];E)\to C([\lambda,T];E)$ is the Lipschitz-continuous mapping introduced in Theorem \ref{PropSalins} and $\gamma_{\lambda, t}$ is the mapping introduced in \eqref{sbm2}. In particular, we have the existence of a unique solution for equation \eqref{mildmain} once we prove that there exists $\epsilon_T>0$ such that the mapping
\[X \in\,\mathcal{H}^{\,p}_{\lambda,t}(E)\mapsto \Lambda_{\lambda,t}^{\e,x}(X):=\mathcal{M}_{\lambda, t}(\sqrt{\e}\,\gamma_{\la, t}(X)+S(\cdot-\lambda)x) \in\,\mathcal{H}^{\,p}_{\lambda,t}(E),\]
is a contraction, for all $\epsilon\leq \epsilon_T$.
According to Theorem \ref{PropSalins} and Proposition \ref{prop stoch conv}, for every  $X_1, X_2 \in \mathcal{H}^{\,p}_{\lambda,t}(E)$ and $0\leq \lambda<t\leq T$ we have
\begin{align*}
&\E\sup_{s\in [\lambda,t]}|\Lambda_{\lambda,t}^{\e,x}(X_1)(s)-\Lambda_{\lambda,t}^{\e,x}(X_2)(s)|_E^p\\[10pt]
&\quad \leq\E\sup_{s\in [\lambda,t]}|\mathcal{M}_{\lambda,t}(\sqrt{\e}\gamma_{\la, t}(X_1)+S(\cdot-\lambda)x)(s)-\mathcal{M}_{\lambda,t}(\sqrt{\e}\,\gamma_{\la, t}(X_2)+S(\cdot-\lambda)x)(s)|^p_E
\\[10pt]
&\quad \quad   \leq L_{\mathcal{M}}^p(T)\, \e^{p/2}\,\E\sup_{s\in [\lambda,t]}|\gamma_{\la, t}(X_1)(s)-\gamma_{\la, t}(X_2)(s)|^p_E\leq L_{\mathcal{M}}^p(T)\,
{L_{\gamma}^p(T,p)} \e^{p/2} |X_1-X_2|^p_{ \mathcal{H}^{\,p}_{\lambda,t}(E)},
\end{align*}
and our claim follows once we take
\[\e_{T,p}<( L_{\mathcal{M}}(T) 
\,{L_{\gamma}(T,p)})^{-2}.\]

In the same way, for every $x_1, x_2 \in\,E$, $\epsilon\leq \epsilon_{T,p}$ and $0\leq \lambda<t\leq T$, we have
\begin{align*}
&|X^{t,x_1}_{\epsilon,\lambda}-X^{t,x_2}_{\epsilon,\lambda}|_{ \mathcal{H}^{\,p}_{\lambda,t}(E)}\\[10pt]
&\quad \quad  =	\left|\mathcal{M}_{\lambda,t}(\sqrt{\e}\,\gamma_{\la, t}(X^{t,x_1}_{\epsilon,\lambda})+S(\cdot-\lambda)x_1)-\mathcal{M}_{\lambda,t}(\sqrt{\e}\,\gamma_{\la, t}(X^{t,x_2}_{\epsilon,\lambda})+S(\cdot-\lambda)x_2)\right|_{\mathcal{H}^{\,p}_{\lambda,t}(E)}\\[10pt]
&\quad \quad \quad \quad \leq L_{\mathcal{M}}(T)\,L_{\gamma,p}(T)\,\left(\sqrt{\epsilon}\,|X^{t,x_1}_{\epsilon,\lambda}-X^{t,x_2}_{\epsilon,\lambda}|_{\mathcal{H}^{\,p}_{\lambda,t}(E)}+|x_1-x_2|_E\right).
\end{align*}
Hence, since $L_{\mathcal{M}}(T)\,L_{\gamma,p}(T)\sqrt{\epsilon}<1$, we get \eqref{bms6}.

\begin{Remark}
	{\em By using similar arguments, we can also prove that
	\begin{equation}
	\label{bms30}
\sup_{\epsilon\leq \epsilon_T}\mathbb{E}\sup_{s \in\,[\lambda,t]}\left|X^{t,x}_{\epsilon,\lambda}(s)\right|_E^p\leq c_{T,p}\,\left(1+|x|_E^p\right).	
	\end{equation}
}
\end{Remark}

 \end{proof}

\section{Proof of Theorem \ref{LDP}}
In this last section we give a proof of the large deviation principle, as stated in Theorem \ref{LDP}. We follow the  method based on weak convergence developed for SPDEs in \cite{bdm}. To this purpose, we need to introduce some notations.
If we define 
\[Q=\sum_{i=1}^{\infty} \lambda_i\, e_i \otimes e_i,\]
according to \eqref{bms1}, the operator $Q$ is the covariance of the noise $W_t$ and  $H_0:=Q^{1/2}(H)$ is its reproducing kernel. We have seen that the space $H_0$, endowed with the scalar product  
$$\langle h,k\rangle_{H_0}=\langle Q^{-1/2} h,Q^{-1/2}k\rangle_{H}=\sum_{i=1}^{\infty} \lambda_i^{-1}\,\langle e_i,h\rangle_{H}\langle e_i,k\rangle_{H}, $$
and the corresponding norm $\vert\cdot\vert_{H_0}$, is a Hilbert space.

For every $t>0$, we denote by $L^2_w(0,t;H_0)$ the space $L^2(0,T;H_0)$, endowed with the weak topology, and for every $M>0$ we introduce the set
\[\mathcal{S}_{t,M}:=\left\{ \varphi \in\,L^2_w(0,t;H_0)\ :\ \vert \varphi\vert_{L^2(0,t;H_0)}\leq M\right\}.\]
Moreover, we introduce the set
\[\Lambda_{t,M}:=\left\{ \varphi \in\,\mathcal{P}_t\ :\ \varphi \in \mathcal{S}_{t,M},\ \mathbb{P}-\text{a.s.}\right\},\]
where $\mathcal{P}_t$ is the set of predictable processes in $L^2(\Omega;L^2(0,t;H_0))$. 

\smallskip

Now, for every $\varphi \in \Lambda_{t,M}$ and $\e>0$, we consider the controlled equation
\begin{equation}
\label{stoch-pde.eps.LDP}
\left\{\begin{array}{l}
\begin{array}{l}
\ds{dX_{\e,\varphi}^{t,x}(s)= \left(A	X_{\e,\varphi}^{t,x}(s)+F(X_{\e,\varphi}^{t,x}(s))+ \sigma(X_{\e,\varphi}^{t,x}(s),v_{t,\e}(s,X_{\e,\varphi}^{t,x}(s)))\varphi(s) \right) \,ds}\\[10pt] 
\ds{ \quad \quad \quad \quad \quad \quad \quad \quad \quad +\sqrt{\e}\,\sigma(X_{\e,\varphi}^{t,x}(s),v_{t,\e}(s,X_{\e,\varphi}^{t,x}(s))) \,dW_s,}\\[8pt]
\ds{X_{\e,\varphi}^{t,x}(0)=x.}
	\end{array}	
\end{array}\right.
\end{equation}
where 
\begin{equation}\label{bms22-bis}v_{t,\e}(\lambda,y):=  \E(g(X^{t,y}_{\e,\lambda}(t)),\ \ \ \ \ \ 0\leq \lambda < t,\ \ \ y \in\,E,\end{equation}
 and $X^{t,y}_{\e,\lambda}$ is the solution of \eqref{mildmain} starting at time $\lambda$ from $y$, that is
 \begin{equation}
\begin{array}{l}
\ds{	X^{t,y}_{\e,\lambda}(s)=S(s-\lambda)y+\int_\lambda^s S(s-r)F(X^{t,y}_{\e,\lambda}(r))dr
}\\[14pt]
\ds{ \quad \quad \quad \quad \quad \quad \quad \quad \quad \quad \quad \quad  +\sqrt{\e}\,\int_\lambda^s S(s-r) \sigma(X^{t,y}_{\e,\lambda}(r),  \E(g(X^{t,y}_{\e,\lambda}(t))|\F_r))\,dW_r.}
\end{array}
\end{equation}
Notice that, according to Theorem \ref{main-ex-un}, for every $0\leq \lambda<t\leq T$, the mapping
$v_{t,\epsilon}(\lambda,\cdot):E\to E$ is Lipschitz continuous, uniformly with respect to $\epsilon\leq \e_{T,1}$. Actually, due to the Lipschitz-continuity of $g$   we have
\begin{align}\label{bms15}
	\vert v_{t,\epsilon}(\lambda,y_1)-v_{t,\epsilon}(\lambda,y_2)\vert_E\leq L_g\,\mathbb{E} \vert X^{t,y_1}_{\e,\lambda}(t)-X^{t,y_2}_{\e,\lambda}(t)\vert_E\leq L_g\,c_{T,1}\,|y_1-y_2|_E.
\end{align}

In what follows, we shall prove the existence and uniqueness of mild solutions for equation \eqref{stoch-pde.eps.LDP}.

\begin{Theorem}
\label{teo8.1}
Under Hypotheses \ref{H1}, \ref{H2} and \ref{H3}, there exists some $\e_{T,p}>0$ such that for every $\varphi \in  \Lambda_{t,M}$  and 
  $x\in E$ and for every $\epsilon\leq\epsilon_{T,p}$ 
there exists a unique mild solution  $X_{\e}^{t,x,\varphi} \in \mathcal{H}^{\,p}_{0,t}(E)$ to equation \eqref{stoch-pde.eps.LDP}. Namely
\begin{equation}\begin{array}{ll}\label{mild con phi}
    X_{\e}^{t,x,\varphi}(s)=\displaystyle\int_0^s S(s-r)F(X_{\e}^{t,x,\varphi}(r))dr + \int_0^s S(s-r)  \sigma(X_{\e}^{t,x,\varphi}(r), v_{t,\e}(r,X_{\e}^{t,x,\varphi}(r)))\varphi(r) \,dr
 \\ \\ \quad\quad\quad\quad\quad\quad\displaystyle+\sqrt{\e}\,\int_0^s S(s-r) \sigma(X_{\e}^{t,x,\varphi}(r), v_{t,\e}(r,X_{\e}^{t,x,\varphi}(r)))\,dW_r+ S(s)x.
\end{array}\end{equation} 
Moreover, we have  \begin{equation}
 \label{bms20}
 \sup_{\e\leq \e_{T,p}}\,\vert 	X_{\e}^{t,x,\varphi}\vert_{\mathcal{H}^{\,p}_{0,t}(E)}\leq c_{t,M,p}\left(1+|x|_E\right).
 \end{equation}
  
 \end{Theorem}
 
 \

Next, as mentioned in Section \ref{main}, for every $\varphi \in\,L^2(0,t;H_0)$ and $x \in\,E$, we introduce the  deterministic controlled problem 
\begin{equation}
\label{sbm2-bis}
\frac {dX}{ds}(s)=A	X(s)+F(X(s))+\sigma(X(s),  v_{t}(s,X(s)))\varphi(s),\ \ \ \ X(0)=x,	\end{equation}
where
\begin{equation}\label{bms22} v_{t}(\rho,y) =  g(Y_\rho^{y}(t)),\ \ \ \ \ \ \rho \in\,[0,t],\ \ \ \ y \in\,E,\end{equation}
 and $Y^{y}_\rho$ is the solution of the equation
\begin{equation}\label{mildlimite}
Y(s)=\int_\rho^s S(s-r)F(Y(r))dr
+ S(s-\rho)y, \ \ \ \ \ s\in\,[\rho,t].
\end{equation}

In what follows, we shall prove that the following result holds.
\begin{Proposition}
\label{prop1}
  For  every  $t>0$ and $\varphi \in\,L^2(0,t;H_0)$ and for every $x \in\,E$, there exists a unique mild solution $X^{t,x,\varphi}\in\,C([0,t];E)$ for equation \eqref{sbm2-bis}.	Namely
 \[X^{t,x,\varphi}(s)=S(s)x+\int_0^s e^{(s-r)A}F(X^{t,x,\varphi}(r))\,dr+\int_0^s  e^{(s-r)A}\sigma(X^{t,x,\varphi}(r),v_{t}(r,X^{t,x,\varphi}(r)))\,\varphi(r)\,dr,\]
 for every $s \in\,[0,t]$. Moreover,
 \begin{equation}\label{bms28}
 |X^{t,x,\varphi}|_{C([0,t];E)}\leq c_{T,M}(1+|x|_E).
 \end{equation}

\end{Proposition}

Once proved Theorem \ref{teo8.1}  and Proposition \ref{prop1}, 
we introduce the action functional 
\begin{equation}
\label{sbm1}
I_{t,x}(X)=\frac 12\,\inf\left\{\int_0^t\vert \varphi(s)\vert_{H_0}^2\,ds\,:\, X(s)=	X^{t,x,\varphi}(s),\ s \in\,[0,t]\right\}.
\end{equation}
Moreover, we introduce
the
following two conditions.
\begin{enumerate}
\item[C1.] Let $\{\varphi_\epsilon\}_{\epsilon>0}$ be an arbitrary family of processes in $ \Lambda_{t,M}$ such that 
\[\lim_{\e\to 0} \varphi_\e=\varphi,\ \ \ \ \text{in distribution in}\ \ \ L_w^2(0,t;H_0).\]
Then we have
\[\lim_{\e\to 0} X_\e^{t,x,\varphi_{\e}}=X^{t,x,\varphi},\ \ \ \ \text{in distribution in}\ \ \ C([0,t],E).\]
\item[C2.] For every $t, R>0$,  the level sets $\Phi_{t,R}=\{I_{t,x}	\leq R\}$ are compact in the space $C([0,t];E)$.
\end{enumerate}

As established in \cite{bdm}, Conditions C1. and C2. ensure that the family $\{X^{t,x}_{\epsilon}\}_{\epsilon\in\,(0,\e_T)}$ satisfies a Laplace Principle on the space $C([0,t];H)$ with  action functional $I_{t,x}$. It is well known that when $I_{t,x}$ possesses compact level sets, the Laplace principle and the large deviation principle  with the same action functional are equivalent. Therefore, since Condition C2. guarantees the compactness of the level sets $\Phi_{t, R}$, verifying Conditions C1. and C2. is effectively equivalent to proving the validity of a large deviation principle for the  family $\{X^{t,x}_{\epsilon}\}_{\epsilon\in\,(0,\e_T)}$, with action functional $I_{t,x}$.

\subsection{Proof of Theorem \ref{teo8.1}}

Given an arbitrary $\tau \in [0,t)$, $\kappa >0$ and $X \in \mathcal{H}^p_{\tau,\tau+\kappa}(E)$, we define 
\begin{align*}
&\gamma_{t,\tau,\tau+\kappa}^{\varphi,\epsilon}(X)(s):=\sqrt{\e} \int_{\tau}^s S(s-r) \sigma(X(r),  v_{t,\e}(r,X(r)))\,dW_r\\[10pt]
&\quad\quad\quad\quad\quad\quad\quad\quad\quad\quad+\int_{\tau}^s S(s-r) \sigma(X(r),  v_{t,\e}(r,X(r)))\,\varphi(r) dr.	
\end{align*}
{\em Step 1.} We claim that there exists $\bar{\kappa}>0$ such that for every $\epsilon \in\,(0,{\e_{T,p}})$ and $\tau \in\,[0,t)$ and every $X_1, X_2 \in\,\mathcal{H}^p_{\tau,\tau+\bar{\kappa}}(E)$ it holds
\begin{equation}\label{bms18} \vert \gamma_{t,\tau,\tau+\bar{\kappa}}^{\varphi,\epsilon}(X_1)-\gamma_{t,\tau,\tau+\bar{\kappa}}^{\varphi,\epsilon}(X_2) \vert_{ \mathcal{H}^p_{\tau,\tau+\kappa}(E)} \leq \frac{1}{2\,L_{\mathcal{M}}(T)}\,
 \vert X_1-X_2 \vert_{\mathcal{H}^p_{\tau,\tau+\bar{\kappa}}(E) }.\end{equation}

If we denote \[\delta\gamma_{t,\tau,\tau+\kappa}^{\varphi,\epsilon}:=\gamma_{t,\tau,\tau+\kappa}^{\varphi,\epsilon}(X_1)-\gamma_{t,\tau,\tau+\kappa}^{\varphi,\e}(X_2),\] for every $\alpha\in (0,1/2)$,
we have
\begin{align*}
&\delta\gamma_{t,\tau,\tau+\kappa}^{\varphi,\epsilon}(s) =  \dfrac{\sin \pi \alpha}{\pi}\, \sqrt{\e}\int_{\tau}^s (s-r)^{\alpha-1}S(s-r)V_{\e,t}^{\alpha}(r)dr\\[10pt]
&\quad \quad \quad \quad \quad \quad \quad \quad +\dfrac{\sin \pi \alpha}{\pi} \int_{\tau}^s (s-r)^{\alpha-1}S(s-r)U_{\e,t}^{\alpha,\varphi}(r)dr,	
\end{align*} 
where
\[\displaystyle V_{\e,t}^{\alpha}(r):=  \int_{\tau}^r (r-\rho)^{-\alpha}S(r-\rho)Y_{\epsilon,t}(\rho)\,dW_{\rho},\ \ \ \ \ \displaystyle \,U_{\e,t}^{\alpha,\varphi}(r):=\int_{\tau}^r (r-\rho)^{-\alpha}S(r-\rho)Y_{\epsilon,t}(\rho)\,\varphi(\rho)\, d\rho,\] and 
\begin{equation*}
  Y_{\epsilon,t}(\rho) :=  \sigma(X_1(\rho),  v_{t,\e}(\rho,X_1(\rho)))- \sigma(X_2(\rho),  v_{t,\e}(\rho,X_2(\rho))).
\end{equation*}
By proceeding as in the proof of Proposition \ref{prop stoch conv}, we can find $\bar{\alpha} \in\,(0,1/2)$ and $\bar{p}\geq 1$ such that for every $p\geq \bar{p}$
\begin{equation}\label{stimagammav_alpha-bis}
  \sup_{s\in [\tau,\tau+\kappa]}  | \delta \gamma_{t,\tau,\tau+\kappa}^{\varphi,\epsilon}(s)|^p_E \leq c_{\bar{\alpha},p,t} \int_\tau^{\tau+\kappa} |V_{\e,t}^{\bar{\alpha}}(r)|^p_{L^p(\mathcal{O})}dr + c_{\bar{\alpha},p,t} \int_\tau^{\tau+\kappa} |U_{\e,t}^{\bar{\alpha},\varphi}(r)|^p_{L^p(\mathcal{O})}\,dr.\end{equation}
    By proceeding as in the proof of \eqref{stimaconvalpha}, we have
$$   \mathbb{E}|V_{\e,t}^{\bar{\alpha}}(r,\xi)|^p \leq c  \,\int_\tau^r(r-\rho)^{-2\bar{\alpha}-\frac{d}{2\zeta}}\,d\rho \int_\tau^r \mathbb{E}\,|Y_{\epsilon, t}(\rho)|^p_E\,d\rho.$$
Therefore, due to the  Lipschitz-continuity of $\sigma$ and $v_{t,\e}$ (see \eqref{bms15}), there exists some $\delta>0$ such that 
    \begin{equation}
        \label{stima V_1}
 \sup_{\epsilon\leq \e_T}\mathbb{E}\int_\tau^{\tau+\kappa}|V_{\e, t}^{\bar{\alpha}}(r)|^p_{L^p(\mathcal{O})}\,dr
    \leq c_T\, \kappa^{\delta} \vert X_2-X_1 \vert_{\mathcal{H}^p_{\tau,\tau+\kappa}(E)}^p.
    \end{equation}
 Next, if we define  
 \[\varphi_i(s):= \lambda_i^{-1/2}\langle\varphi(s),e_i\rangle_H,\ \ \ \ \ i\in \mathbb{N},\ \ \ \ \ s \in\,(0,t),\]
  we have
$$U_{\e,t}^{\bar{\alpha},\varphi}(r,\xi)=\sum_{i=1}^{\infty} \sqrt{\la_i}\int_\tau^r (r-\rho)^{-\bar{\alpha}}S(r-\rho)(Y_{\epsilon,t}(\rho)e_i) (\xi)\,\varphi_i(\rho) d \rho,\quad \quad \quad r\in [0,t],\ \  \ \xi\in \Oc.$$
Since \[\sum_{i=1}^{\infty} |\varphi_i(\rho)|^2=|\varphi(\rho)|^2_{H_0},\] according to Hypothesis \ref{H3} we have
\begin{align}\label{stima v 1-bis}
\nonumber &|U_{\e,t}^{\bar{\alpha},\varphi}(r,\xi)|^2 \leq c  \left(\int_{\tau}^r \left(\sum_{i=1}^{\infty}\la_i (r-\rho)^{-2\bar{\alpha}}\left|S(r-\rho)(Y_{\epsilon,t}(\rho)e_i) (\xi)\right|^2 \right)^{1/2}|\varphi(\rho)|_{H_0} d\rho\right)^2  \\[10pt]
\nonumber&\leq c  \,\sum_{i=1}^{\infty} \la_i |e_i|_{L^{\infty}(\mathcal{O})}^ {2/\theta}\int_{\tau}^r (r-\rho)^{-2\bar{\alpha}}\left|S(r-\rho)(Y_{\epsilon,t}(\rho)e_i) (\xi)\right|^2  |e_i|_{L^{\infty}(\mathcal{O})}^ {-2/\theta}d\rho  \int_{\tau}^r |\varphi(\rho)|_{H_0}^2 d\rho\\[10pt]
\nonumber & \leq c  \,  M^2 \left(\sum_{i=1}^{\infty} \la_i^{\theta} |e_i|_{L^{\infty}(\mathcal{O})}^ {2}\right)^{1/\theta}\int_{\tau}^r (r-\rho)^{-2\bar{\alpha}} \left( \sum_{i=1}^{\infty}\left|S(r-\rho)(Y_{\epsilon,t}(\rho)e_i) (\xi)\right|^{ 2\zeta} |e_i|_{L^{\infty}(\mathcal{O})}^ {-2 (\zeta-1)}\right)^{1/\zeta}  d\rho\\[10pt]
 & \leq c_M  \,  \int_{\tau}^r (r-\rho)^{-2\bar{\alpha}} A_{\e,t}^{1/\zeta}(r,\rho,\xi)  d\rho,
\end{align}
where $ \frac{1}{\theta} + \frac{1}{\zeta}=1$ and 
$ A_\epsilon(r,\rho,\xi)$ is defined as in \eqref{def-A} with $Y_t$ replaced by $Y_{\e,t}$.  By proceeding as in the proof of \eqref{bms17}, we have 
$$ A_{\epsilon,t}(r,\rho,\xi)\leq \,  (r-\rho)^{-\frac{d}{2}}|Y_{\epsilon,t}(\rho)|_E^{2\zeta},$$
and if we plug the inequality above into \eqref{stima v 1-bis}, we get
\[|U_{\e,t}^{\bar{\alpha},\varphi}(r,\xi)|^p\leq c_M\,\left(\int_{\tau}^r (r-\rho)^{-2\bar{\alpha}-\frac d{2\zeta}} |Y_{\e,t}(\rho)|_E^2  d\rho\right)^{p/2}.\]
Now, due to the Lipschitz-continuity of $\sigma$ and \eqref{bms15},
we have
\[\sup_{\e\leq \e_T}\mathbb{E}\,|Y_{\e,t}(\rho)|_E^p\leq c_t\,\vert X_1-X_2\vert^p_{\mathcal{H}^{\,p}_{\tau,\tau+\kappa}(E)},\]
so that
\begin{align*}
&\int_\tau^{\tau+\kappa}|U^{\bar{\alpha},\varphi}_{\e,t}(r)|_{L^p(\mathcal{O})}^p\,dr \leq c_M  \, \int_{\tau}^{\tau+\kappa} r^{-2\bar{\alpha}-\frac{d}{2\zeta}}\,dr  \int_{\tau}^{\tau+\kappa} |Y_{\epsilon,t}(\rho)|^p_E\,d\rho\\[10pt]
&\quad \quad \quad \quad \quad \quad \quad \quad \quad \quad \quad \quad \quad \quad \quad \quad \leq c_{M,T}\, \kappa^{2-2\bar{\alpha}-\frac{d}{2\zeta}}\vert X_1-X_2\vert^p_{\mathcal{H}^{\,p}_{\tau,\tau+\kappa}(E)}. 	
\end{align*}

 In view of the inequality above and \eqref{stima V_1}, from \eqref{stimagammav_alpha-bis} we get
\begin{equation} \label{contrazione con phi}
     \mathbb{E} \sup_{s\in [\tau,\tau+\kappa]}| \delta \gamma_{t,\tau,\tau+\kappa}
     ^{\varphi,\epsilon}(s)|^p_E \leq   c_{M,T}\, \kappa^{2-2\bar{\alpha}-\frac{d}{2\zeta}} \vert X_2-X_1 \vert_{\mathcal{H}^{\,p}_{\tau,\tau+\kappa}(E)}^p,
\end{equation}
so that we can find $\bar{\kappa}=\bar{\kappa}_{M,T}>0$, independent of $\tau$, $x$ and $\e\leq \e_{T,p}$, such that \eqref{bms18} holds for every $X_1, X_2 \in\,\mathcal{H}^{\,p}_{\tau,\tau+\kappa}(E)$.

{\em Step 2.} To complete the proof of the existence and uniqueness of solution, we start by considering   equation \eqref{mild con phi}  in $\mathcal{H}^{\,p}_{0,\bar{\kappa}}(E)$, which can be written as
\[X=\mathcal{M}_{0,\bar{\kappa}}(\gamma^{\varphi,\e}_{t,0,\bar{\kappa}}(X)+S(\cdot)x).\]
 Taking into account of \eqref{bms18} and Theorem \ref{PropSalins} we immediately deduce that  the map
\[X\in\,\mathcal{H}^{\,p}_{0,\bar{\kappa}}(E)\mapsto  \mathcal{M}_{0,\bar{\kappa}}(\gamma^{\varphi,\e}_{t,0,\bar{\kappa}}(X)+S(\cdot-\bar{\kappa})x) \in\,\mathcal{H}^{\,p}_{\bar{\kappa},2\,\bar{\kappa}}(E)\]
is a contraction and its fixed point is the unique solution $X_{\e}^{t,x,\varphi}$ of equation \eqref{mild con phi} in $[0,\bar{\kappa}]$. Next, we consider the mapping
\[X\in\,\mathcal{H}^{\,p}_{\bar{\kappa},2\bar{\kappa}}(E)\mapsto  \mathcal{M}_{\bar{\kappa}, 2\bar{\kappa}}(\gamma_{t,\bar{\kappa},2\bar{\kappa}}^{\varphi,\epsilon}(X)+S(\cdot)X_{\e}^{t,x,\varphi}(\bar{\kappa})) \in\,\mathcal{H}^{\,p}_{0,\bar{\kappa}}(E),\]
which, due again to \eqref{bms18} and Theorem \ref{PropSalins} has a unique fixed point.  
Thus, we can proceed iteratively considering the equation in all intervals $[i\bar{\kappa},(i+1) \bar{\kappa}]$, for $i=0,\ldots,[t/\bar{\kappa}]$, and we get a unique solution for equation \eqref{mild con phi} in $[0,t]$.

{\em Step 3.} In order to prove \eqref{bms20}, we note that
\begin{align*}
	&\vert X^{t,x,\varphi}_{\epsilon}\vert_{\mathcal{H}^p_{0,\bar{\kappa}}(E)}\leq \vert \mathcal{M}_{0,\bar{\kappa}}(\gamma^{\varphi,\e}_{t,0,\bar{\kappa}}(X^{t,x,\varphi}_{\epsilon})+S(\cdot)x)-\mathcal{M}_{0,\bar{\kappa}}(\gamma^{\varphi,\e}_{t,0,\bar{\kappa}}(0))\vert_{\mathcal{H}^p_{0,\bar{\kappa}}(E)}\\[10pt]
	&\quad \quad \quad \quad \quad \quad \quad \quad \quad \quad \quad +\vert \mathcal{M}_{0,\bar{\kappa}}(\gamma^{\varphi,\e}_{t,0,\bar{\kappa}}(0))\vert_{\mathcal{H}^p_{0,\bar{\kappa}}(E)},
\end{align*}
so that, thanks to the Lipschitz-continuity of $\mathcal{M}_{0,\bar{\kappa}}$, \eqref{bms21} and \eqref{bms18}, we get
\begin{align*}
	& \vert X^{t,x,\varphi}_{\epsilon}\vert_{\mathcal{H}^p_{0,\bar{\kappa}}(E)}\\[10pt]
	&\quad \leq L_{\mathcal{M}}(T)\,\left(\vert \gamma^{\varphi,\e}_{t,0,\bar{\kappa}}(X^{t,x,\varphi}_{\epsilon})-\gamma^{\varphi,\e}_{t,0,\bar{\kappa}}(0)\vert_{\mathcal{H}^p_{0,\bar{\kappa}}(E)}+|x|_E\right)+c_T\left(|F(\gamma^{\varphi,\e}_{t,0,\bar{\kappa}}(0))|_E+\vert \gamma^{\varphi,\e}_{t,0,\bar{\kappa}}(0)\vert_E\right)\\[10pt]
	& \quad \quad \leq \frac 12 \vert X^{t,x,\varphi}_{\epsilon}\vert_{\mathcal{H}^p_{0,\bar{\kappa}}(E)}+L_{\mathcal{M}}(T)\,|x|_E+c_T\left(|F(\gamma^{\varphi,\e}_{t,0,\bar{\kappa}}(0))|_E+\vert \gamma^{\varphi,\e}_{t,0,\bar{\kappa}}(0)\vert_E\right).
\end{align*}
This implies 
\[\vert X^{t,x,\varphi}_{\epsilon}\vert_{\mathcal{H}^p_{0,\bar{\kappa}}(E)}\leq 2\,L_{\mathcal{M}}(T)\,|x|_E+2\,c_T\,|F(\gamma^{\varphi,\e}_{t,0,\bar{\kappa}}(0))|_E.\]
As the same can be repeated in all intervals $[i\bar{\kappa},(i+1)\bar{\kappa}]$, for all $i=0,\ldots,[t/\bar{\kappa}]$, we conclude that \eqref{bms20} holds.

\subsection{Proof of Proposition \ref{prop1}} 

{\em Step 1.} If $v_t$ is the function introduced in \eqref{bms22}, then  the mapping
\[y \in\,E\mapsto v_t(\rho,y) \in\,E,\]
 is Lipschitz-continuous, uniformly with respect to $\rho \in\,[0,t]$.
 
 Actually, due to the Lipschitz-continuity of $g$, we have
 \[|v_t(\rho,y_1)-v_t(\rho,v_2)|_E\leq L_g\,|Y^{y_1}_\rho(t)-Y^{y_2}_\rho(t)|_E.\]
With the notation introduced in Theorem \ref{PropSalins}, we have
\[ Y^{y_1}_\rho(t)-Y^{y_2}_\rho(t)=\mathcal{M}_{\rho,t}(S(\cdot-\rho)y_1)(t)-\mathcal{M}_{\rho,t}(S(\cdot-\rho)y_2)(t),\]
so that, according to Theorem \ref{PropSalins}, we have
\begin{equation}
\label{bms23prima}
|v_t(\rho,y_1)-v_t(\rho,v_2)|_E\leq L_g L_{\mathcal{M}}(T)\,|y_1-y_2|_E\ \ \ \ \ \rho \in\,[0,t].	
\end{equation}

\smallskip

{\em Step 2.} By using again the notations introduced in Theorem \ref{PropSalins}, we have
that $X^{t,x,\varphi}$ is a mild solution of equation \eqref{sbm2-bis} in the interval $[\tau,\tau+\kappa]$ if and only if
it is a fixed point of the mapping
\begin{equation}\label{bms23}X \in\,C([\tau,\tau+\kappa];E)\mapsto \mathcal{M}_{\tau,\tau+\kappa}(\gamma_{t,\tau,\tau+\kappa}^\varphi(X)+S(\cdot-\tau)x) \in\,C([\tau,\tau+\kappa];E),\end{equation}
where
\[\gamma_{t,\tau,\tau+\kappa}^\varphi(X)(s):= \int_\tau^s  e^{(s-r)A}\sigma(X(r),v_{t}(r,X(r)))\,\varphi(r)\,dr,\ \ \ \ \ \ s \in\,[\tau,\tau+\kappa].\]
By proceeding with arguments analogous to those used in Step 1. of the proof of Theorem \ref{teo8.1}, we have that for every $X_1, X_2 \in\,C([\tau,\tau+\kappa];E)$
\[|\gamma_{t,\tau,\tau+\kappa}^\varphi(X_1)-\gamma_{t,\tau,\tau+\kappa}^\varphi(X_2)|_{C([\tau,\tau+\kappa];E)}\leq c\,\kappa^{\delta}\,|X_1-X_2|_{C([\tau,\tau+\kappa];E)},\]
for some $\delta>0$, so that, due to the Lipschitz-continuity of $\mathcal{M}_{\tau,\tau+\kappa}$ we can find $\bar{\kappa}$, independent of $x \in\,E$ and $\tau$ such that the mapping introduced in \eqref{bms23} is a contraction and hence admits a unique fixed point. 

\smallskip

{\em Step 3.} By proceeding with a bootstrap argument as in Step 2. of  the proof of Theorem \ref{teo8.1}, we conclude that there exists a unique mild solution $X^{t,x,\varphi} \in\,C([0,t];E)$ for equation 
\eqref{sbm2-bis}. Moreover, by using arguments analogous to those used above, we have that estimate \eqref{bms28} holds.

\subsection{The validity of condition C1.}
Thanks to Skorohod's theorem we rephrase  condition C1. in the following way. Let $(\bar{\Omega}, \bar{\mathcal{F}},\bar{\mathbb{P}})$ be a probability space and let $\{W^\e_r\}_{r \in\,[0,t]}$ be a  Wiener process, with covariance $Q$, defined on such  probability space and corresponding to the filtration $\{\bar{\mathcal{F}}_t\}_{t\geq 0}$. Moreover, let $\{\bar{\varphi}_\e\}_{\e>0}$ and $\bar{\varphi}$ be  $\{\bar{\mathcal{F}}_t\}_{t\geq 0}$-predictable processes  in $\Lambda_{t,M}$, such that  
\[\mathcal{L}(\bar{\varphi}_{\epsilon}, \bar{\varphi}, W^\epsilon)=\mathcal{L}(\varphi_\e,\varphi,W),\] and
 \[\lim_{\e\to 0}\bar{\varphi}_\e=\bar{\varphi}\ \ \ \text{ in } L_w^2(0,T;H_{0}),\ \ \ \ \bar{\mathbb{P}}-\text{a.s.}\]
 Then, if $\bar{X}^{t,x, \bar{\varphi}_\e}_{\epsilon}$ is the solution of an equation analogous to \eqref{stoch-pde.eps.LDP}, with $\varphi_\e$ and $W_r$ replaced respectively by $\bar{\varphi}_\e$ and $W_r^\epsilon$, we need to prove that 
 \begin{equation}
 \label{ca155}
 \lim_{\e\to 0} \bar{X}^{t,x, \bar{\varphi}_\e}_{\epsilon}=\bar{X}^{t,x, \bar{\varphi}}\ \ \ \text{in }C([0,t];E),\ \ \ 
\text{in distribution.}\end{equation}

 \smallskip
 
With the notations introduced in the proofs of Theorem \ref{teo8.1} and Proposition \ref{prop1}, for every $x\in\,E$ and $\tau \in\,[0,t)$ and $\kappa>0$, we have
 \[\begin{array}{l}
\ds{\bar{X}^{t,x, \bar{\varphi}_\e}_{\epsilon}(s)-\bar{X}^{t,x, \bar{\varphi}}(s)	}\\[10pt]
\ds{\quad \quad =\mathcal{M}_{\tau,\tau+\kappa}\left(\gamma^{\bar{\varphi}_\e,\e}_{t,\tau,\tau+\kappa}(\bar{X}^{t,x, \bar{\varphi}_\e}_{\epsilon})+S(\cdot)x\right)(s)-\mathcal{M}_{\tau,\tau+\kappa}\left(\gamma^{\bar{\varphi}}_{t,\tau,\tau+\kappa}(\bar{X}^{t,x, \bar{\varphi}})+S(\cdot)\right)(s),}
\end{array}\]
for every $s \in\,[\tau,\tau+\kappa]$. Thus,  the Lipschitz continuity of $\mathcal{M}_{\tau,\tau+\kappa}$ gives
\[|\bar{X}^{t,x, \bar{\varphi}_\e}_{\epsilon}-\bar{X}^{t,x, \bar{\varphi}}|_{C([\tau,\tau+\kappa];E)}\leq L_{\mathcal{M}}(T)\,|\gamma^{\bar{\varphi}_\e,\e}_{t,\tau,\tau+\kappa}(\bar{X}^{t,x, \bar{\varphi}_\e}_{\epsilon})-\gamma^{\bar{\varphi}}_{t,\tau,\tau+\kappa}(\bar{X}^{t,x, \bar{\varphi}})|_{C([\tau,\tau+\kappa];E)}.\]
Now, due to \eqref{contrazione con phi}, we have
\[\begin{array}{l}
\ds{\bar{\mathbb{E}}\,|\gamma^{\bar{\varphi}_\e,\e}_{t,\tau,\tau+\kappa}(\bar{X}^{t,x, \bar{\varphi}_\e}_{\epsilon})-\gamma^{\bar{\varphi}}_{t,\tau,\tau+\kappa}(\bar{X}^{t,x, \bar{\varphi}})|_{C([\tau,\tau+\kappa];E)}}\\[14pt]
\ds{\quad \quad \quad\leq \bar{\mathbb{E}}\,|\gamma^{\bar{\varphi}_\e,\e}_{t,\tau,\tau+\kappa}(\bar{X}^{t,x, \bar{\varphi}_\e}_{\epsilon})-\gamma^{\bar{\varphi}_\e,\e}_{t,\tau,\tau+\kappa}(\bar{X}^{t,x, \bar{\varphi}})|_{C([\tau,\tau+\kappa];E)}}\\[14pt]
\ds{\quad \quad \quad \quad \quad \quad\quad \quad \quad+\bar{\mathbb{E}}\,|\gamma^{\bar{\varphi}_\e,\e}_{t,\tau,\tau+\kappa}(\bar{X}^{t,x, \bar{\varphi}})-\gamma^{\bar{\varphi}}_{t,\tau,\tau+\kappa}(\bar{X}^{t,x, \bar{\varphi}})|_{C([\tau,\tau+\kappa];E)}}	\\[14pt]
\ds{\leq c_{M,T} \,\kappa^{\delta}\,\bar{\mathbb{E}}\,|\bar{X}^{t,x, \bar{\varphi}_\e}_{\epsilon}-\bar{X}^{t,x, \bar{\varphi}}|_{C([\tau,\tau+\kappa);E)}+\bar{\mathbb{E}}\,|\gamma^{\bar{\varphi}_\e,\e}_{t,\tau,\tau+\kappa}(\bar{X}^{t,x, \bar{\varphi}})-\gamma^{\bar{\varphi}}_{t,\tau,\tau+\kappa}(\bar{X}^{t,x, \bar{\varphi}})|_{C([\tau,\tau+\kappa];E)},}
\end{array}\]
for some $c_{M,T}$ independent of $\tau \in\,[0,t)$,  $x \in\,E$, and $\epsilon\leq \e_T:=\e_{T,1}.$
In particular, we can pick $\bar{\kappa}>0$ such that 
\begin{align*}
&\bar{\mathbb{E}}\,|\bar{X}^{t,x, \bar{\varphi}_\e}_{\epsilon}-\bar{X}^{t,x, \bar{\varphi}}|_{C([\tau,\tau+\bar{\kappa}];E)}\\[10pt]
& \quad \quad \leq 
\frac 12\,\bar{\mathbb{E}}\,|\bar{X}^{t,x, \bar{\varphi}_\e}_{\epsilon}-\bar{X}^{t,x, \bar{\varphi}}|_{C([\tau,\tau+\bar{\kappa}];E)} +L_{\mathcal{M}}(T)\,\bar{\mathbb{E}}\,|\gamma^{\bar{\varphi}_\e,\e}_{t,\tau,\tau+\kappa}(\bar{X}^{t,x, \bar{\varphi}})-\gamma^{\bar{\varphi}}_{t,\tau,\tau+\kappa}(\bar{X}^{t,x, \bar{\varphi}})|_{C([\tau,\tau+\kappa];E)}.	
\end{align*}
Therefore, if we prove that 
\begin{equation}\label{bms33}
\lim_{\e\to 0}	\bar{\mathbb{E}}|\gamma^{\bar{\varphi}_\e,\e}_{t,0,\bar{\kappa}}(\bar{X}^{t,x, \bar{\varphi}})-\gamma^{\bar{\varphi}}_{t,0,\bar{\kappa}}(\bar{X}^{t,x, \bar{\varphi}})|_{C([0,\bar{\kappa}];E)}=0,
\end{equation}
we get $
 \lim_{\e\to 0} \bar{X}^{t,x, \bar{\varphi}_\e}_{\epsilon}=\bar{X}^{t,x, \bar{\varphi}}$ in $L^1(\bar{\Omega}, \bar{\mathbb{P}},C([0,\bar{\kappa}];E)$, and 
\eqref{ca155} holds in the interval $[0,\bar{\kappa}]$.

We have
\begin{align*}
	&\gamma^{\bar{\varphi}_\e,\e}_{t,0,\bar{\kappa}}(\bar{X}^{t,x, \bar{\varphi}})(s)-\gamma^{\bar{\varphi}}_{t,0,\bar{\kappa}}(\bar{X}^{t,x, \bar{\varphi}})(s)=\sqrt{\e} \int_{0}^s S(s-r) \sigma(\bar{X}^{t,x, \bar{\varphi}}(r),  v_{t,\e}(r,\bar{X}^{t,x, \bar{\varphi}}(r)))\,dW_r\\[10pt]
	&\quad +\int_{0}^s S(s-r) \left(\sigma(\bar{X}^{t,x, \bar{\varphi}}(r),  v_{t,\e}(r,\bar{X}^{t,x, \bar{\varphi}}(r)))-\sigma(\bar{X}^{t,x, \bar{\varphi}}(r),  v_{t}(r,\bar{X}^{t,x, \bar{\varphi}}(r)))\right)\,\bar{\varphi}_\epsilon(r) dr\\[10pt]
	&\quad+\int_{0}^s S(s-r) \sigma(\bar{X}^{t,x, \bar{\varphi}}(r),  v_{t}(r,\bar{X}^{t,x, \bar{\varphi}}(r)))\,(\bar{\varphi}_\epsilon(r) -\bar{\varphi}(r))\,dr=:I^\e_1(s)+I^\e_2(s)+I^\e_3(s).
\end{align*}

By proceeding as in the proof of \eqref{contrazione con phi}, due to \eqref{bms28} we have
\begin{equation}\label{bms50}
\bar{\mathbb{E}}\,|I^\e_1|_{C([0,\bar{\kappa}];E)}\leq c_{T,M}\sqrt{\e}\left(1+|\bar{X}^{t,x, \bar{\varphi}}|_{C([0,\bar{\kappa}];E)}\right)\leq c_{T,M}\sqrt{\e}\left(1+|x|_E\right).	
\end{equation}
As for the second term $I^\e_2$, by proceeding as in Step 1. in the proof of Theorem \ref{teo8.1}, 
since for every $\e \in\,(0,\e_T)$
\[\int_0^t \vert \bar{\varphi}_\epsilon(r)\vert_{H_0}^2\,dr\leq M^2,\ \ \ \ \bar{\mathbb{P}}-\text{a.s.}\] 
we have
 \begin{align*}
	&|I^\e_2(s)|_E\leq c_{T,M}\int_0^s|\sigma(\bar{X}^{t,x, \bar{\varphi}}(r),  v_{t,\e}(r,\bar{X}^{t,x, \bar{\varphi}}(r)))-\sigma(\bar{X}^{t,x, \bar{\varphi}}(r),  v_{t}(r,\bar{X}^{t,x, \bar{\varphi}}(r)))|_E\,dr,
\end{align*} 
and then, due to the Lipschitz-continuity of $\sigma$ we get that for $p$ sufficiently large
\begin{align*}
	&|I^\e_2(s)|_E\leq c_{T,M}\,L_\sigma\int_0^s|v_{t,\e}(r,\bar{X}^{t,x, \bar{\varphi}}(r)))-v_{t}(r,\bar{X}^{t,x, \bar{\varphi}}(r))|_E\,dr.
	\end{align*}
Therefore, if we can prove that
\begin{equation}
\label{bms31}
\lim_{\e\to 0}	\int_0^s|v_{t,\e}(r,\bar{X}^{t,x, \bar{\varphi}}(r)))-v_{t}(r,\bar{X}^{t,x, \bar{\varphi}}(r))|_E\,dr=0,
\end{equation}
we conclude that
\begin{equation}\label{bms32}
\lim_{\e\to 0} 	|I^\e_2|_{C([0,\bar{\kappa}];E)}=0,\ \ \ \ \ \bar{\mathbb{P}}-\text{a.s.}
\end{equation}

Recalling how $v_{t,\epsilon}$ and $v_t$ were defined in \eqref{bms22} and \eqref{bms22-bis} respectively, we have
\begin{align*}
&|v_{t,\e}(r,y)-v_{t}(r,y)|_E\leq L_g\,\bar{\mathbb{E}}\,|X^{t,y}_{\epsilon,r}(t)-Y^y_r(t)|_E	\\[10pt]
&\quad \quad \quad \quad =L_g\,|\mathcal{M}_{0,t}(\sqrt{\e}\,\gamma_{0,t}(X^{t,y}_{\epsilon,r}))+S(t-\cdot)y)(t)-\mathcal{M}_{0,t}(S(t-\cdot)y)(t)|_E\\[10pt]
&\quad \quad \quad \quad \quad \quad \quad \quad \leq L_gL_{\mathcal{M}}(T)\,\sqrt{\epsilon}\,\bar{\mathbb{E}}\sup_{s \in\,[0,t]}|\gamma_{0,t}(X^{t,y}_{\e,r})(s)|_E.
\end{align*}
Hence, by using arguments analogous to those used in the proof of Proposition \ref{prop stoch conv}, thanks to \eqref{bms30} we get
\[|v_{t,\e}(r,y)-v_{t}(r,y)|_E\leq \sqrt{\e}\,L_gL_{\mathcal{M}}(T)c_{T}\left(1+|y|_E\right).\]
This implies
\[\int_0^s|v_{t,\e}(r,\bar{X}^{t,x, \bar{\varphi}}(r)))-v_{t}(r,\bar{X}^{t,x, \bar{\varphi}}(r))|_E\,dr\leq \sqrt{\e}\,L_gL_{\mathcal{M}}(T)c_{T}\left(1+\sup_{r \in\,[0,s]}|\bar{X}^{t,x, \bar{\varphi}}(r)|_E\right),\]
so that, thanks to \eqref{bms28}, we obtain \eqref{bms31}  and \eqref{bms32} follows. 

Next, concerning $I^\e_3$, for every $\alpha \in\,(0,1/2)$ we have
\[I^\e_3(s)=\frac{\sin \pi \alpha}{\pi}\int_0^s
(s-r)^{\alpha-1}S(s-r)Y_{\alpha,\epsilon}(r)\,dr,\]
where
\[Y_{\alpha,\epsilon}(r):=\int_0^r (r-\rho)^{-\alpha}S(r-\rho)h(\rho)\left(\bar{\varphi}_\epsilon(\rho) -\bar{\varphi}(\rho)\right)\,d\rho,\]
and
\[h(\rho):=\sigma(\bar{X}^{t,x, \bar{\varphi}}(\rho),  v_{t}(\rho,\bar{X}^{t,x, \bar{\varphi}}(\rho))).\]
This implies that if we fix $\eta>0$ and $p\geq 1$ such that
\[\frac \eta 2+\frac 1p<\alpha,\]
we have
\begin{align*}
&\vert I^\e_3(s)\vert_{W^{\eta,p}(\mathcal{O})}\leq c_\alpha\int_0^s (s-r)^{\alpha-1-\frac\eta 2}\vert Y_{\alpha,\epsilon}(r)\vert_{L^p(\mathcal{O})}\,dr\\[10pt]
&\leq c_\alpha\left(\int_0^s (s-r)^{(\alpha-1-\frac\eta 2)\frac p{p-1}}\,dr\right)^{\frac{p-1}p}\left(\int_0^s\vert Y_{\alpha,\epsilon}(r)\vert^p_{L^p(\mathcal{O})}\,dr\right)^{\frac 1p}=: c_{\a,p}(s)\left(\int_0^s\vert Y_{\alpha,\epsilon}(r)\vert^p_{L^p(\mathcal{O})}\,dr\right)^{\frac 1p}.
	\end{align*}
In particular, if we pick $p>d/\eta$ we have 
\[\sup_{s \in\,[0,t]}\,\vert I^\epsilon_3\vert _E\leq c \sup_{s \in\,[0,t]}\,\vert I^\e_3\vert_{W^{\eta,p}(\mathcal{O})}\leq c_{\a,p}(t)\left(\int_0^t\vert Y_{\alpha,\epsilon}(r)\vert^p_{L^p(\mathcal{O})}\,dr\right)^{\frac 1p}.\]
This means that if we show that
\begin{equation}\label{fine1}
\lim_{\e\to 0}\int_0^t\bar{\mathbb{E}}\,\vert Y_{\alpha,\epsilon}(r)\vert^p_{L^p(\mathcal{O})}\,dr=0,	
\end{equation}
we can conclude that 
\begin{equation}
	\label{terzo}
	\lim_{\e\to0}  \bar{\mathbb{E}}\vert I^\epsilon_3\vert_{C([0,\bar{\kappa}];E)}=0.
\end{equation}

For every $(r,\xi) \in\,[0,t]\times \mathcal{O}$, we have  
\begin{align*}
	&Y_{\alpha,\e}(r,\xi)=\sum_{i=1}^\infty \int_{0}^r(r-\rho)^{-\alpha}  S(r-\rho)(h e_i)(\xi)\,\langle\bar{\varphi}_\epsilon(\rho) -\bar{\varphi}(\rho),e_i\rangle_H\,d\rho,
\end{align*}
so that
\begin{align}\label{fine3}
\nonumber&\vert Y_{\alpha,\e}(r,\xi)\vert^2\leq \sum_{i=1}^\infty \lambda_i\int_{0}^r(r-\rho)^{-2\alpha}  \vert S(r-\rho)(h e_i)(\xi)\vert^2\,d\rho\,\sum_{i=1}^\infty\frac{1}{\lambda_i}\int_0^t\vert \langle\bar{\varphi}_\epsilon(\rho) -\bar{\varphi}(\rho),e_i\rangle_H\vert^2\,d\rho\\[10pt]
&\quad \quad \quad \quad \quad \quad \quad \quad \leq c\,M^2\,\sum_{i=1}^\infty \lambda_i\int_{0}^r(r-\rho)^{-2\alpha}  \vert S(r-\rho)(h e_i)(\xi)\vert^2\,d\rho.\end{align}
By proceeding as in the proof of Proposition \ref{prop stoch conv}, we can show that 
\[\sum_{i=1}^\infty \lambda_i\int_{0}^r(r-\rho)^{-2\alpha}  \vert S(r-\rho)(h e_i)(\xi)\vert^2\,d\rho\leq c\,\int_0^r(r-\rho)^{-2\alpha-\frac{d}{2\zeta}}\vert h(\rho)\vert_E^2\,d\rho,\]
where, we recall,
$\zeta=\theta/(\theta-1)$ and $\theta$ is the constant introduced in Hypothesis \ref{H3}.
Since both $\sigma$ and $v_t$ have linear growth, according to \eqref{bms28} we have
\[\sup_{\rho \in\,[0,t]}\vert h(\rho)\vert_E\leq c_{\,T,M}\left(1+\vert x\vert_E\right).\]
Thus, if we pick $\bar{\alpha}>0$ small enough so that 
$2\bar{\alpha}+d/(2\zeta)<1$, we conclude that
\begin{equation}\label{fine4}\sum_{i=1}^\infty \lambda_i\int_{0}^r(r-\rho)^{-2\alpha}  \vert S(r-\rho)(h e_i)(\xi)\vert^2\,d\rho\leq c_{\,T,M}\left(1+\vert x\vert_E\right).\end{equation}
This implies that if we take an arbitrary $\delta>0$, we can find $i_\delta \in\,\mathbb{N}$ such that
\[\vert Y_{\alpha,\e}(r,\xi)\vert\leq c  \sum_{i\leq i_\delta} \big|\int_{0}^r(r-\rho)^{-\bar{\alpha}}  S(r-\rho)(h e_i)(\xi)\,\langle\bar{\varphi}_\epsilon(\rho) -\bar{\varphi}(\rho),e_i\rangle_H\,d\rho\big|+\delta.\]

Since $\bar{\alpha}<1/2$, we have that, for every fixed $(r,\xi) \in\,[0,t]\times \mathcal{O}$, the mapping
\[(\rho,\eta) \in\,[0,t]\times \mathcal{O}\mapsto \mathbb{I}_{[0,r]}(\rho)(r-\rho)^{-\alpha} S(r-\rho)(h e_i)(\xi)e_i(\eta) \in\,\mathbb{R},\]
belongs to $L^2(0,t;H)$, and hence, as $\bar{\varphi_\epsilon}$ converges weakly to $\bar{\varphi}$ in $L^2(0,t;H)$, due to the arbitrariness of $\delta>0$ we conclude that 
\[\lim_{\e\to 0} \vert Y_{\alpha,\e}(r,\xi)\vert=0.\]
Moreover, thanks 
to \eqref{fine3} and \eqref{fine4}, we have 
\[\vert Y_{\alpha,\e}(r,\xi)\vert^2 \leq c_{\,T,M},\ \ \ \ \ (r,\xi) \in\,[0,t]\times \mathcal{O},\]
and the dominated convergence theorem implies \eqref{fine1}, which, as we have seen above, implies \eqref{terzo}.

Finally, if we combine together  \eqref{bms50}, \eqref{bms32} and \eqref{terzo} we conclude that \eqref{bms33} holds  in $[0,\bar{\kappa}]$. As the same arguments can be repeated in all intervals $[i\bar{\kappa},(i+1)\bar{\kappa}]$, we have that  
\[\lim_{\e\to 0} \bar{\mathbb{E}}|\bar{X}^{t,x, \bar{\varphi}_\e}_{\epsilon}-\bar{X}^{t,x, \bar{\varphi}}|_{C([0,t];E)}=0,\]
and
\eqref{ca155} follows.
This concludes the proof of the validity of condition C1.

\subsection{The validity of condition C2.} Condition C2. is a special case of condition C1. and follows from analogous arguments. For more details see \cite[Subsection 8.4]{CGT}.

\bibliographystyle{plain}

\end{document}